\documentclass{scrartcl}
\usepackage[T1]{fontenc}
\usepackage{amsmath}
\usepackage{amsthm}
\usepackage{amsfonts}
\usepackage{amssymb}
\usepackage{amsxtra}
\usepackage[shortlabels]{enumitem}
\usepackage{mathrsfs}
\usepackage{hyperref}
\usepackage{mathtools}
\usepackage[scr=boondox,  
            cal=esstix]   
           {mathalpha}
\usepackage{xfrac}  
\usepackage{cite}
\usepackage{tensor}
\hypersetup{
breaklinks=true,
colorlinks=true,
linkcolor=blue,
citecolor=blue,
urlcolor=blue,
filecolor=blue,
}

\numberwithin{equation}{section} 
\setkomafont{title}{\normalfont}

\newcommand{\R}{\mathbb{R}}

\newcommand{\Z}{\mathbb{Z}}
\newcommand{\N}{\mathbb{N}}
\newcommand{\calN}{\mathscr{N}}
\newcommand{\calS}{\mathcal{S}}
\newcommand{\calX}{\mathscr{X}}
\newcommand{\tin}{\text{in }}

\newcommand{\ton}{\text{on }}

\newcommand{\tfor}{\text{for }}

\DeclareMathOperator*{\esssup}{ess\,sup}
\newcommand{\np}[1]{(#1)}
\newcommand{\nb}[1]{[#1]}
\newcommand{\bp}[1]{\big(#1\big)}

\newcommand{\Bp}[1]{\bigg(#1\bigg)}

\newcommand{\set}[1]{{\{#1\}}}

\newcommand{\setc}[2]{{\{#1 : #2\}}}
\newcommand{\setcl}[2]{{\bigl\{#1 : #2\bigr\}}}

\newcommand{\norm}[1]{\lVert#1\rVert}

\newcommand{\snorm}[1]{{\lvert #1 \rvert}}
\newcommand{\snorml}[1]{{\bigl\lvert #1 \big\rvert}}
\newcommand{\snormL}[1]{{\Bigl\lvert #1 \Big\rvert}}
\newcommand{\CS}[1]{C^{#1}}
\newcommand{\LS}[1]{L^{#1}}
\newcommand{\LSloc}[1]{L_{\mathrm{loc}}^{#1}}
\newcommand{\WS}[2]{W^{#1,#2}}
\newcommand{\WSloc}[2]{W_\mathrm{loc}^{#1,#2}}

\newcommand{\DS}[2]{D^{#1,#2}}

\newcommand{\tracespss}{\mathfrak T^{q}(\Sigma)}
\newcommand{\tracesppp}{\mathfrak T^{p,q}_\perp(\torus\times\Sigma)}
\newcommand{\tracesptp}{\mathfrak T^{p,q}(\torus\times\Sigma)}
\DeclareMathOperator{\Div}{div}

\newcommand{\ddt}{\frac{{\mathrm d}}{{\mathrm d}t}}

\newcommand{\dx}{{\mathrm d}x}
\newcommand{\dr}{{\mathrm d}r}

\newcommand{\dtheta}{{\mathrm d}\theta}
\newcommand{\ds}{{\mathrm d}s}
\newcommand{\dt}{{\mathrm d}t}
\newcommand{\dy}{{\mathrm d}y}

\newcommand{\dS}{{\mathrm d}S}
\newcommand{\vvel}{v}
\newcommand{\vpres}{\mathscr{p}}

\newcommand{\zvel}{z}
\newcommand{\wvel}{w}
\newcommand{\wpres}{\mathcal{q}}
\newcommand{\uvel}{u}
\newcommand{\upres}{\mathcal{p}}
\newcommand{\uvels}{u_0}
\newcommand{\upress}{\mathcal{p}_0}
\newcommand{\uvelp}{u_\perp}
\newcommand{\upresp}{\mathcal{p}_\perp}
\newcommand{\ouvel}{\overline{\uvel}}
\newcommand{\oupres}{\overline{\mathcal{p}}}

\newcommand{\wakefct}{\mathscr{s}_\zeta}
\newcommand{\per}{\mathcal{T}}
\newcommand{\torus}{\mathbb{T}}
\newcommand{\proj}{\mathscr P}
\newcommand{\projcompl}{\mathscr P_\perp}
\newcommand{\fsolss}{\Gamma^\zeta_0}
\newcommand{\fsolpp}{\Gamma^\zeta_\perp}
\newcommand{\fsolpres}{P}
\DeclareMathOperator{\bop}{\mathcal{B}_R}
\theoremstyle{plain}
\newtheorem{thm}{Theorem}[section]

\newtheorem{lem}[thm]{Lemma}
\newtheorem{prop}[thm]{Proposition}

\theoremstyle{remark}
\newtheorem{rem}[thm]{Remark}

\begin{document}
\title{Approximation of time-periodic flow past a translating body by flows in bounded domains}

\author{Thomas Eiter%
\footnote{Freie Universit\"at Berlin, Department of Mathematics and Computer Science, Arnimallee 14, 14195 Berlin, Germany}
\textsuperscript{\,,}%
\footnote{%
Weierstrass Institute for Applied Analysis and Stochastics,
Mohrenstra\ss{}e 39, 10117 Berlin, Germany.
Email: {\texttt{thomas.eiter@wias-berlin.de}}}%
\and
Ana Leonor Silvestre%
\footnote{%
CEMAT and Department of Mathematics, Instituto Superior T\'ecnico, Universidade de
Lisboa, Av. Rovisco Pais 1, 1049-001 Lisboa, Portugal.
Email: {\texttt{ana.silvestre@math.tecnico.ulisboa.pt}}}%
}

\date{}
\maketitle

\begin{abstract}
We consider a time-periodic incompressible three-dimensional Navier-Stokes flow past a translating rigid body. In the first part of the paper, we establish the existence and uniqueness of strong solutions in the exterior domain $\Omega \subset {\mathbb R}^3$ that satisfy pointwise estimates for both the velocity and pressure. The fundamental solution of the time-periodic Oseen equations plays a central role in obtaining these estimates. The second part focuses on approximating this exterior flow within truncated domains $\Omega \cap B_R$, incorporating appropriate artificial boundary conditions on $\partial B_R$. For these bounded domain problems, we prove the existence and uniqueness of weak solutions. Finally, we estimate the error in the velocity component as a function of the truncation radius $R$, showing that, as $R \to \infty$, the velocities of the truncated problems converge, in an appropriate norm, to the velocity of the exterior flow.
\end{abstract}

\noindent
\textbf{MSC2020:} 35Q30, 76D05, 76D07, 35B10
\\
\noindent
\textbf{Keywords:}   Time-periodic solutions, incompressible Navier--Stokes flows, exterior domains, Oseen flows, fundamental solution, artificial boundary conditions, approximation, truncation error.

\tableofcontents

\section{Introduction}

Consider an incompressible viscous flow around a rigid body translating with a constant velocity $\zeta\in\R^3\setminus\set{0}$. For simplicity and without loss of generality, we take the kinematic viscosity of the fluid to be equal to 1. To describe the motion of the fluid, we use a reference frame attached to the solid. Additionally, we assume the fluid to be subject to an external body force and a 
distribution of velocities along the fluid-solid boundary, both time-periodic of period $\per>0$. Under these conditions, the motion of the fluid is governed by the following equations 
\begin{equation}\label{eq:nstp}
\left\{\ \begin{aligned}
\partial_t\uvel-\Delta\uvel-\zeta\cdot\nabla\uvel+\uvel\cdot\nabla\uvel+\nabla\upres&=f
&&
\tin\torus\times\Omega, \\
\nabla \cdot \uvel&=0 
&&
\tin\torus\times\Omega, \\
\uvel&=h
&&
\ton\torus\times\Sigma, \\
\lim_{\snorm{x}\to\infty}\uvel(t,x)&=0 
&&
\tfor t\in\torus.
\end{aligned}
\right.
\end{equation}
Here and throughout the paper, $\Omega\subset\R^3$ denotes the exterior domain occupied by the liquid, while $\Sigma\coloneqq\partial\Omega$ represents the common boundary between $\Omega $ and the compact set corresponding to the rigid body. We assume that $0 \in {\mathbb R}^3 \setminus \overline{\Omega}$. Since we are interested in time-periodic flows, the torus group $\torus\coloneqq\R/\per\Z$ serves as the time axis in system~\eqref{eq:nstp}, so that all functions therein 
are time-periodic with period $\per>0$. The functions
$\uvel\colon\torus\times\Omega\to\R^3$ and $\upres\colon\torus\times\Omega\to\R$ 
represent the unknown velocity field  and scalar pressure, respectively. 

In the context of applications, a crucial question is how to numerically solve the exterior problem \eqref{eq:nstp}. Truncating the fluid domain in order to discretize the equations using, for instance, finite elements necessarily introduces artificial boundaries, which must be chosen so as to ensure the well-posedness of the mathematical model and the numerical stability of the simulations. Prescribing the so-called ``do-nothing'' condition \cite{heywood1996artificial,rannacher2012} on the artificial boundaries arises naturally in the variational formulation after multiplication of the term $-\Delta\uvel + \nabla\upres$ with a test function and integration by parts. However, as shown in \cite{Braack2014,LanzHron2020}, this Neumann condition does not guarantee the well-posedness of the resulting boundary value problem for the Navier-Stokes equations. In \cite{PD1997}, the question of how to numerically solve the Dirichlet problem for the Stokes system in the exterior of a three-dimensional bounded Lipschitz domain is addressed using a modified ``do-nothing'' condition on the outer boundary of a truncated domain. A similar idea was subsequently exploited for more complex fluid models in \cite{DKN2021,DeuringKracmarExtStatNSFApprBddDom2004}, and we adopt it in this work, as described below.

\medskip

{\textbf{\emph{Formulation of the problem}}}. Our aim is to investigate how to approximate solutions $(\uvel,\upres)$ to system~\eqref{eq:nstp}, formulated in the unbounded domain $\Omega$,
by solutions $(\uvel_R,\upres_R)$ to problems posed in bounded domains 
$\Omega_R=\setc{x\in\Omega}{\snorm{x}<R}$
for $R>0$ sufficiently large.
More precisely, we consider solutions 
$(\vvel,\vpres)=(\uvel_R,\upres_R)$ 
to the truncated problems
\begin{equation}\label{eq:nstp.pert}
\left\{\ \begin{aligned}
\partial_t\vvel-\Delta\vvel-\zeta\cdot\nabla\vvel+\vvel\cdot\nabla\vvel+\nabla \vpres &=f
&&
\tin\torus\times\Omega_R, \\
\nabla \cdot \vvel&=0 
&&
\tin\torus\times\Omega_R, \\
\vvel&=h
&&
\ton\torus\times\Sigma, \\
\bop(\vvel,\vpres)&=0
&&
\ton\torus\times\partial B_R,
\end{aligned}\right.
\end{equation}
where $\bop$ is a suitable boundary operator. The artificial boundary condition $\bop(\vvel,\vpres)=0$ on $\torus\times\partial B_R$
must be selected to ensure both the well-posedness of the resulting mixed boundary value problem and the convergence of $\uvel_R$ to $\uvel$ 
as $R\to\infty$ in an appropriate norm. Our choice
\begin{equation}
\bop(\vvel,\vpres)(t,x)
= \frac{x}{R}\cdot\Bp{\nabla\vvel(t,x)- \vpres(t,x) {\mathsf I} - \frac{1}{2}\vvel(t,x)\otimes\vvel(t,x)}
+  \frac{1 + \wakefct(x)}{R}\vvel(t,x)
\label{abc}
\end{equation}
where $\wakefct(x):= \left[|\zeta| |x| + (\zeta\cdot x )\right] / 2$, is inspired by ~\cite{DeuringKracmarExtStatNSFApprBddDom2004}.
The present work is a generalization to the time-periodic case of the results obtained in~\cite{DeuringKracmarExtStatNSFApprBddDom2004} for the steady problem (see also~\cite{DKN2021} for a linearized steady flow around a rotating and translating body). Note that the operator $\bop$ defined in \eqref{abc} contains the pseudo-stress tensor $\widetilde{\mathsf{T}}( \vvel,\vpres) =\nabla \vvel - \vpres \mathsf I$.
However, all results in this paper remain valid  if $\widetilde{\mathsf{T}}$ is replaced by the classical Cauchy stress tensor $\mathsf{T}( \vvel ,\vpres)=\nabla\vvel+\nabla\vvel^\top-\vpres\mathsf I$. Here, the gradient of a vector-valued function of several variables is the transpose of the Jacobian matrix: $(\nabla v)_{ij}=\frac{\partial v_j}{\partial x_i},$ $i,j=1,2,3$.

To present the main results of the paper, we introduce additional notation and recall basic properties of the relevant function spaces and operators.

\medskip

{\textbf{\emph{Notations}}}.  Throughout the paper, we will consistently use the same font style to represent scalar, vector, and tensor-valued functions.  Standard notations $L^p(\mathcal O)$, $W^{k,p}(\mathcal O)$ and $H^k(\mathcal O)$ for suitable sets $\mathcal O$ 
will be adopted for Lebesgue and Sobolev spaces, and we occasionally write 
$\norm{\cdot}_{p;D}\coloneqq\norm{\cdot}_{L^p(\mathcal O)}$
and $\norm{\cdot}_{k,p;D}\coloneqq\norm{\cdot}_{W^{k,p}(\mathcal O)}$
for corresponding norms. 
We further introduce homogeneous 
Sobolev spaces by denoting $u\in D^{k,p}(\mathcal O)$ 
if and only $u$ is locally integrable with $\nabla^k u\in L^p(\mathcal O)$.
We further introduce the homogeneous 
By ${\mathcal D}({\mathbb T})$ we denote  the class of real-valued, infinitely differentiable, ${\mathcal T}$-periodic functions.

By ${\mathsf I}\in\R^{3\times 3}$ we denote the three-dimensional identity matrix.
We denote the Dirac delta distributions on ${\mathbb R}^3$, $\mathbb T$, and $\mathbb Z$ by $\delta_{{\mathbb R}^3}$, $\delta_{{\mathbb T}}$ and $\delta_{{\mathbb Z}}$, respectively. 
Here $\torus:=\R/\per\Z$, where the period $\per>0$ is fixed throughout the paper.
The whole-space problem associated with~\eqref{eq:nstp} will be formulated in the locally compact abelian group $G := {\mathbb T} \times  {\mathbb R}^3$, and the Dirac delta distribution on $G$, $\delta_G$, will be used to define the fundamental solution of the time-periodic problem. In the context of the exterior problem $\Omega$, the symbol $\delta_\Sigma$ will denote the Dirac delta distribution with support $\Sigma=\partial\Omega \subset {\mathbb R}^3$. 
By ${\mathscr S}'({\mathbb R}^3)$ and ${\mathscr S}'(G)$ we will denote the spaces of tempered distributions over ${\mathbb R}^3$ and $G$, respectively.

If $X$ is a Banach space, we denote  by $L^r({\mathbb T};X)$ the space of all Bochner measurable functions $u:{\mathbb T} \to X$ such that $\| u\|_{L^r({\mathbb T};X)}:=\left(\frac{1}{\mathcal T}\int_0^{\mathcal T} \|u(t)\|_X^r\ dt \right)^{\frac 1r}< \infty,$ for $1 \leq r < \infty$, and $\| u\|_{L^\infty({\mathbb T};X)}:=\operatorname{ess\,sup}_{t \in [0,\mathcal T]}\|u(t)\|_X < \infty,$ for $r = \infty$. 
We denote by $C({\mathbb T};X)$ the space of continuous functions 
$f:\torus\to X$, which corresponds to the continuous functions
$f:[0,\mathcal T] \to X$ that satisfy $f(0) = f(\mathcal T)$. 

We will utilize a precise decomposition of the solution into a steady-state component and a purely periodic component, as proposed and employed in ~\cite{Ky16,EiterKyed2017,GaldiKyedtpflowViscLiquidBody18,EiterKyedEstTPFundSol2018,EiterShibata_FlowPastBodyOscBdry}. Specifically,  time-periodic functions $v \colon\torus\to X$ are split into a steady-state part $v_0=\proj v$ and a purely periodic part $v_\perp=\projcompl v$, where the projections $\proj$ and $\projcompl$ are defined by
\begin{equation}
\label{eq:tp.split}
\proj v
\coloneqq \int_\torus v(t)\,\dt
= \frac{1}{\per}\int_0^\per v(t)\,\dt,
\qquad
\projcompl v \coloneqq v - \proj v.
\end{equation}

To specify the class of admissible boundary traces of strong solutions to~\eqref{eq:nstp} we define 
\[
\tracesptp\coloneqq
\setcl{\vvel|_{\torus\times\Sigma}}{
\vvel\in\LS{p}(\torus;\WS{2}{q}(\Omega)^3), \,\partial_t\vvel\in\LS{p}(\torus;\LS{q}(\Omega)^3)}
\]
for $p,q\in(1,\infty)$,
and we equip this space with the norm
\[
\begin{aligned}
\norm{h}_{\tracesptp}
\coloneqq \inf\setcl{
\norm{\vvel}_{\LS{p}(\torus;\WS{2}{q}(\Omega))}
+\norm{\partial_t\vvel}_{\LS{p}(\torus;\LS{q}(\Omega))}
}{
h=\vvel|_{\torus\times\Sigma}
}.
\end{aligned}
\]
This function space can be decomposed into 
spaces of steady-state and of purely periodic functions, given by 
\[
\tracespss
\coloneqq\setcl{\proj h}{h\in\tracesptp},
\qquad 
\tracesppp
\coloneqq\setcl{\projcompl h}{h\in\tracesptp}.
\]
Then $\tracespss$ coincides with the Sobolev--Slobodeckij space 
$\WS{2-1/q}{q}(\Sigma)$.
Similarly, one can identify $\tracesptp$ and $\tracesppp$ 
with suitable interpolation spaces,
which are of Triebel--Lizorkin and Besov type.
Since these involved constructions are not necessary for our approach, 
we omit them here.

When studying the exterior problem, to quantify the decay of functions in a suitable way, 
we introduce the weight function
\[
\begin{aligned}
\nu^\alpha_\beta(x;\zeta)
&\coloneqq \snorm{x}^\alpha (1+\wakefct(x))^\beta,
\qquad
\wakefct(x)\coloneqq \frac{1}{2}\big[|\zeta|\,|x| +\zeta\cdot x \big]
\end{aligned}
\]
for $\alpha,\beta\in\R$,
and the corresponding weighted norms
\[
\begin{aligned}
\norm{v}_{\infty,\nu^\alpha_\beta(\cdot;\zeta);D}
&\coloneqq \esssup_{x\in D}\nu^\alpha_\beta(x;\zeta)\snorm{v(x)},
\\
\norm{v}_{\infty,\nu^\alpha_\beta(\cdot;\zeta);\torus\times D}
&\coloneqq \esssup_{(t,x)\in \torus\times D} \nu^\alpha_\beta(x;\zeta)\snorm{v(t,x)},
\end{aligned}
\]
where $D\subset\R^3$ is an open set.
When $\beta=0$, we simply write $\nu^\alpha\coloneqq\nu^\alpha_0(\cdot,\zeta)$, 
so that
$\nu^\alpha(x):=\snorm{x}^\alpha$.

\medskip

{\textbf{\emph{Main results}}}. The paper's first main result, Theorem \ref{thm:strongsol}, establishes the existence and uniqueness of strong solutions to problem \eqref{eq:nstp}, assuming that $f\in \LSloc{1}(\torus\times\Omega)^3$ and $h\in \tracesptp$ with
$$
\norm{\proj f}_{\infty,\nu^{5/2}_1(\cdot;\zeta); \Omega}
+\norm{\projcompl f}_{\infty,\nu^{4}(\cdot;\zeta); \torus\times \Omega}
+\norm{h}_{\tracesptp}
$$
sufficiently small. The corresponding solution $(\uvel,\upres)$ possesses the same decay as the time-periodic fundamental solution, exhibiting an anisotropic decay determined by the steady-state part of the fundamental solution. The decay rates of the pressure and the purely periodic part of the velocity field depend on whether the total flux $\Phi$ across $\Sigma$, defined by 
$$\Phi(t):= \int_\Sigma h(t,x)\cdot n\,\dS(x),$$ 
is constant in time.

Subsequently, we consider problem~\eqref{eq:nstp.pert} incorporating the artificial boundary condition \eqref{abc} on the outer boundary of the truncated spatial domain. The second main result, Theorem \ref{thm:weakR}, establishes the existence and conditional uniqueness of weak solutions $(\uvel_R,\upres_R)$ to \eqref{eq:nstp.pert}-\eqref{abc}, under weaker assumptions on the regularity of the boundary data and provided that $\| \Phi\|_{\infty,\mathbb T}$ is small.

Assuming the validity of the earlier well-posedness results, as the third main result of the paper, Theorem \ref{thm:error}, we prove the following convergence for the gradient of the velocity and for its trace on the artificial boundaries:
\begin{equation}
\norm{\nabla\uvel-\nabla\uvel_R}_{L^{2}(\torus\times\Omega_R)} + \| u -  u_R \|_{L^2({\mathbb T}\times \partial B_R)} 
\leq C R^{-\sfrac{1}{2}}.
\label{converg}
\end{equation}

\medskip

{\textbf{\emph{Structure of the paper}}}. A review of the fundamental solutions of the steady-state and time-periodic Oseen equations, along with the estimates useful for our study, is provided in Section \ref{aux}. Section \ref{exterior} addresses the well-posedness of the exterior problem~\eqref{eq:nstp}, including the precise spatial decay of the velocity and pressure. In Section \ref{truncateddomain}, we establish existence and conditional uniqueness of weak solutions to the system~\eqref{eq:nstp.pert}--\eqref{abc}. Finally, the estimate for the truncation error of the velocity field is derived in Section \ref{error}.

\section{Fundamental solutions}
\label{aux}

In this section, we introduce the fundamental solution of the time-periodic Oseen equations,
\begin{equation}\label{eq:oseentp}
\left\{\ \begin{aligned}
\partial_t\uvel-\Delta\uvel-\zeta\cdot\nabla\uvel+\nabla\upres&=f
&&
\tin\torus\times\Omega, \\
\nabla \cdot \uvel&=0 
&&
\tin\torus\times\Omega, \\
\uvel&=h
&&
\ton\torus\times\Sigma.
\end{aligned}\right.
\end{equation} 

We begin by recalling several fundamental solutions for steady problems. In ${\mathbb R}^3$, the fundamental solution of the Laplace operator $-\Delta$ is given by
\begin{equation} E(x)=  \frac{1}{4\pi |x|}, \label{la-funda}
\end{equation}
that is, $- \Delta E = \delta_{{\mathbb R}^3}$ in ${\mathscr S}'({\mathbb R}^3)$. The fundamental solution of the 3D Stokes system is the pair $(\Gamma_{0}^{0}, P) \in {\mathscr S}'({\mathbb R}^3)^{3 \times 3} \times {\mathscr S}'({\mathbb R}^3)^{3}$ given by (see, for example, \cite{GaldiBook2011})
$$
(\Gamma_{0}^{0}, P)(x) = \left(\frac{1}{8\pi  |x|}\left({\mathsf I} + \hat{x} \otimes \hat{x} \right),\frac{1}{4\pi |x|^2}\hat{x} \right)
$$
where $\hat{x} := x / |x|$ ($x \in {\mathbb R}^3\setminus \{0\}$), and the pressure component satisfies
\begin{equation}
 P(x) = - \nabla E(x). 
 \label{PE}
\end{equation}
The fundamental solution of the 3D Oseen system has the same pressure part, $P(x) = - \nabla E(x)$, and the velocity component is given by (see \cite{GaldiBook2011,ALSOseenFS2019})
\begin{equation}
\begin{aligned}
 \Gamma^{\zeta}_{0}(x) & =
 \frac{1}{4\pi |x|}\exp\left(-\wakefct(x)\right)  {\mathsf I} -  \frac{|\zeta|}{16 \pi  \wakefct(x)}  \exp\left(-\wakefct(x)\right)   \left(  \hat{x} + \hat{\zeta}  \right) \otimes \left( \hat{x} + \hat{\zeta} \right)  
\\
&\quad  -  \frac{1-\exp(-\wakefct(x))}{8\pi |x| \wakefct(x)}  \left( {\mathsf I}  - \hat{x} \otimes \hat{x} \right) 
\\
&\quad +  \frac{|\zeta|}{16 \pi}   \frac{1-\exp(-\wakefct(x))}{\wakefct(x)^2}   \left(  \hat{x} + \hat{\zeta}  \right) \otimes \left( \hat{x} + \hat{\zeta} \right).
 \end{aligned}
\label{ose}
\end{equation}

In the time-periodic case, the fundamental solutions can be identified as solutions to a system of partial differential equations on $G$. Following \cite{Ky16,EiterKyed2017,GaldiKyedtpflowViscLiquidBody18}, the fundamental solution of the Stokes ($\zeta=0$) or Oseen ($\zeta\not =0$) equations is a pair $(\Gamma^\zeta, Q) \in {\mathscr S}'(G)^{3 \times 3} \times {\mathscr S}'(G)^{3}$ satisfying
\begin{equation}
\left\{ \ 
\begin{aligned}
\partial_t \Gamma^\zeta -  \Delta \Gamma^\zeta  +  \nabla Q - (\zeta \cdot \nabla) \Gamma^\zeta & = \delta_G {\mathsf I} ,  \\
\nabla \cdot \Gamma^\zeta & = 0.  
\end{aligned}
\right.
\label{periodicfundsol}
\end{equation}
The pressure component is given by (recall~\eqref{PE}) 
$$
Q = \delta_{{\mathbb T}} \otimes P, 
$$
meaning $Q(t,x) = \delta_{{\mathbb T}}(t)P(x)$. As in the Stokes case \cite{Ky16}, the velocity part $\Gamma^\zeta$ is a sum of the steady-state Oseen fundamental solution and a purely time-periodic remainder satisfying good integrability and pointwise decay estimates. The pressure part, as in the steady regime, is identical to that of the Stokes case, that is, $Q$ is independent of $\zeta$. The velocity component $\Gamma^\zeta$  admits the following decomposition
$$
\Gamma^\zeta = 1_{{\mathbb T}} \otimes \Gamma_{0}^\zeta + \Gamma_{\perp}^\zeta, 
$$
with $\Gamma_{0}^\zeta$ the velocity part of the steady fundamental solution, defined in~\eqref{ose}, and $\Gamma_{\perp}^\zeta$ the purely periodic part of $\Gamma^\zeta$, defined by
$$
\Gamma_{\perp}^\zeta(t,x) = {\mathcal F}^{-1}_G \left[  \frac{1-\delta_{\mathbb Z}(k)}{|\xi|^2 + i \left(    \frac{2 \pi k}{\mathcal T}  - \zeta \cdot \xi \right)} \left( {\mathsf I} - \hat{\xi} \otimes \hat{\xi}\right)\right],
$$
where ${\mathcal F}_G: {\mathscr S}'(G) \to {\mathscr S}'(\hat{G})$, $\hat{G}:={\mathbb Z} \times {\mathbb R}^3$, is the Fourier transform on the group $G$.

We recall pointwise estimates of the different parts of the fundamental solution.

\begin{prop}
\label{prop:decay.fundsol}
For all $\alpha\in\N_0^3$, $r\in[1,\infty)$ and $\varepsilon>0$
there are $C_1,C_2>0$ 
such that for all $x\in\R^3$ with $\snorm{x}\geq \varepsilon$ 
it holds 
\begin{align}
\snorm{D_x^\alpha \fsolss(x)} 
&\leq  C_1\nu^{-1-\snorm{\alpha}/2}_{-1-\snorm{\alpha}/2}(x;\zeta),
\label{est:fsolss.decay}
\\
\norm{D_x^\alpha \fsolpp(\cdot,x)}_{\LS{r}(\torus)} 
&\leq C_2 \nu^{-3-\snorm{\alpha}}(x).
\label{est:fsolpp.decay}
\end{align}
Here $C_1=C_1(\alpha,r,\varepsilon)>0$ and
$C_2=C_2(\alpha,r,\varepsilon,\theta)>0$ are independent of $\zeta$ and $\per$
if $\per\snorm{\zeta}^2\leq \theta$.
\end{prop}

\begin{proof}
See \cite[Lemma 3.2]{Farwig_habil} and 
\cite[Theorem 1.1]{EiterKyedEstTPFundSol2018}.
For the uniformity of the estimates, see also 
\cite[Theorem 5.8]{EiterShibata_FlowPastBodyOscBdry}.
\end{proof}

When using the anisotropic estimates of $\fsolss$,
we will come across integrals of the form
\begin{equation}
{\mathcal J}_R(a,b):= \int_{\partial B_R} |x|^{-a}(1+ \wakefct(x))^{-b}\, \dS(x) = \int_{\partial B_R} \nu^{-a}_{-b}(x;\zeta)\, \dS(x)
\label{eq:Jab}
\end{equation}
for $a,b\geq0$, $R>0$.
In \cite[Lemma 2.3]{Farwig_statextOseenNSE_92} (see also \cite[Lemma 3.1]{DeuringKracmarExtStatNSFApprBddDom2004}), using polar coordinates, it is shown that
\begin{equation}
{\mathcal J}_R(a,b) \leq C(b) R^{2 - a -\min\{1,b\}} , \quad b \not= 1.
\label{far}
\end{equation}
When dealing with $\fsolpp$,
we shall need the following integrability properties of $\fsolpp$.
\begin{prop}
We have
\begin{align}
&\forall q\in\Bp{1,\frac{5}{3}}:\quad \fsolpp\in\LS{q}(\torus\times\R^3)^{3\times 3},
\label{eq:fsolpp.intfct}\\
&\forall q\in\bigg[1,\frac{5}{4}\bigg):\quad 
\partial_j\fsolpp\in\LS{q}(\torus\times\R^3)^{3\times 3}\quad(j=1,2,3).
\label{eq:fsolpp.intgrad}
\end{align}
If $0<\snorm{\zeta}\leq\zeta_0$ for some $\zeta_0>0$, 
the respective $\LS{q}$-norm can be bounded uniformly in $\zeta$.
\end{prop}

\begin{proof}
See
\cite[Theorem 1.1]{EiterKyedEstTPFundSol2018} and
\cite[Theorem 5.8]{EiterShibata_FlowPastBodyOscBdry}.
\end{proof}

\section{Existence in the exterior domain}
\label{exterior}

We return to the problem \eqref{eq:nstp} in the exterior domain $\Omega$ and show existence of solutions with suitable decay properties.
In what follows, we use the decomposition of time-periodic functions into a steady-state part $f_0=\proj f$ and a purely periodic part $f_\perp=\projcompl f$
introduced in~\eqref{eq:tp.split}. 
 Our aim is to prove:
\begin{thm} \label{thm:strongsol}
Let $\Omega\subset\R^3$ be an exterior domain with $\CS{2}$-boundary 
$\Sigma=\partial\Omega$.
Let $\zeta_0>0$ and $p,q\in(1,\infty)$.
Then there exists $\varepsilon>0$ such that for all 
$f\in \LSloc{1}(\torus\times\Omega)^3$ and $h\in \tracesptp$ 
satisfying
\begin{equation}
\norm{\proj f}_{\infty,\nu^{5/2}_1(\cdot;\zeta); \Omega}
+\norm{\projcompl f}_{\infty,\nu^{3+\delta}(\cdot;\zeta); \torus\times \Omega}
+\norm{h}_{\tracesptp}
\leq\varepsilon^2
\label{restdata}
\end{equation}
for some $\delta\in(0,1]$,  
and for all $\zeta\in\R^3\setminus\set{0}$ with $\snorm{\zeta}\leq \zeta_0$,
there exists a 
unique strong solution $(\uvel,\upres)$ to \eqref{eq:nstp}
satisfying
\[
\uvel\in\LS{p}(\torus;\DS{2}{q}(\Omega)^3),
\qquad
\partial_t\uvel\in\LS{p}(\torus;\LS{q}(\Omega)^3),
\qquad
\upres\in\LS{p}(\torus;\DS{1}{q}(\Omega))
\]
and
\begin{equation}
\begin{aligned}
&\norm{\nabla^2\uvel,\partial_t\uvel,\nabla\upres}_{\LS{p}(\torus;\LS{q}(\Omega))}
+\norm{\proj\uvel}_{\infty,\nu^1_1(\cdot;\zeta); \Omega}
+\norm{\nabla\proj\uvel}_{\infty,\nu^{3/2}_{3/2}(\cdot;\zeta); \Omega}
+ \norm{\proj\upres}_{\infty,\nu^{2}; \Omega} 
\\
&\qquad\qquad
+\norm{\projcompl\uvel}_{\infty,\nu^{2}; \torus\times \Omega}
+\norm{\nabla\projcompl\uvel}_{\infty,\nu^{3}; \torus\times \Omega}
+\norm{\projcompl\upres}_{\infty,\nu^1; \torus\times \Omega} 
\leq \varepsilon.
\end{aligned}
\label{estssol}
\end{equation}
If the boundary data satisfies
\begin{equation}
\label{eq:constantflux}
\forall t\in\torus: \quad 
\int_\Sigma \partial_t h(t,x)\cdot n\,\dS(x)=0,
\end{equation}
then 
\begin{equation}
\norm{\projcompl\uvel}_{\infty,\nu^{3}; \torus\times \Omega}
+\norm{\nabla\projcompl\uvel}_{\infty,\nu^{3+\delta}; \torus\times \Omega}
+ \norm{\projcompl\upres}_{\infty,\nu^2; \torus\times \Omega} 
\leq \varepsilon. 
\label{eq:pres.constantflux}
\end{equation}
\end{thm}

\begin{rem}
\label{fluxcte}
Condition~\eqref{eq:constantflux} means that
the total boundary flux 
\begin{equation}\label{eq:flux}
\Phi(t)\coloneqq \int_\Sigma h(t,x)\cdot n\,\dS(x)
\end{equation}
is constant in time,
that is, 
$\ddt \Phi \equiv 0$.
If this is satisfied,
then the decay rate of the pressure is $\snorm{x}^{-2}$,
while for non-constant total flux,
the pressure only decays like $\snorm{x}^{-1}$
as $\snorm{x}\to\infty$.
Similarly, the decay of the purely periodic part of the velocity field is faster in this case.
This observation is in accordance with~\cite{EiterSilvestreReprForm},
where the decay rates for time-periodic weak solutions to~\eqref{eq:nstp} were derived. 
\end{rem}

A similar existence result was obtained 
in~\cite{EiterShibata_FlowPastBodyOscBdry},
but with a different spatial decay rate of the solutions.
Since the decay assumptions on external forces considered in~\cite[Theorem 4.2]{EiterShibata_FlowPastBodyOscBdry}
are weaker,
the decay rates of the derived solutions are slower as well.
In contrast, for solutions 
established in Theorem~\ref{thm:strongsol}
the velocity field $\uvel$ has the same decay as the time-periodic fundamental solution,
namely the anisotropic decay determined by the steady-state part.
Moreover, the purely periodic velocity field $\projcompl\uvel$
decays faster than the steady-state part $\proj\uvel$,
and the decay rate is improved if~\eqref{eq:constantflux} is satisfied,
that is, for constant total boundary flux.
For $\delta=1$, this pointwise behavior coincides with the decay observed for 
weak solutions when $f$ has compact support,
see also~\cite{EiterSilvestreReprForm},
and can thus be considered the optimal decay rate.

Firstly, we will study a linearized version of problem \eqref{eq:nstp}, with focus on specific pointwise estimates. Then, a fixed point argument yields the result of Theorem \ref{thm:strongsol}.

\subsection{Linear theory}

To prove Theorem~\ref{thm:strongsol},
we first study the associated linear problem \eqref{eq:oseentp}. For pointwise decay estimates of the velocity field $\uvel=\uvels+\uvelp$
split into steady-state and purely periodic parts, 
we extend the velocity and pressure to zero outside the domain $\Omega$ and employ the representation formulas 
(see \cite{ALSOseenFS2019,EiterSilvestreReprForm})
\begin{align}
&\begin{aligned}
\uvels
& = \fsolss \ast_{\R^3} \nb{  f_0 \chi_\Omega + n \cdot \widetilde{\mathsf T}(\uvels,\upress) \delta_{\Sigma} 
 + (\zeta \cdot n)  h_0 \delta_{\Sigma}}
 \\
& \qquad
 + \fsolss \ast_{\R^3} \nabla  \cdot  \nb{  (n \otimes h_0 )\delta_{\Sigma} } 
-P \ast_{\R^3} \nb{n\cdot h_0\delta_\Sigma},
\end{aligned}
\label{eq:repr.vel.lin.s}
\\
&\begin{aligned}
\uvelp
& = \fsolpp \ast_{G} \nb{  f_\perp \chi_\Omega + n \cdot \widetilde{\mathsf T}(\uvelp,\upresp) \delta_{\Sigma} 
 + (\zeta \cdot n)  h_\perp \delta_{\Sigma}} 
 \\
& \qquad
 + \fsolpp \ast_G \nabla \cdot \nb{  (n \otimes h_\perp )\delta_{\Sigma} } 
-Q\ast_G \nb{n\cdot  h_\perp\delta_\Sigma},
\end{aligned}
\label{eq:repr.vel.lin.p}
\end{align}
where 
$\widetilde{\mathsf T}(v,q)
=\nabla v - q \mathsf I$
denotes the Cauchy pseudo-stress tensor for the velocity-pressure pair $(v,q)$. The corresponding formulas for the pressure are given by
\begin{align}
&\begin{aligned}
\upress
&= c_0 + P \ast_{\R^3} \nb{  f_0 \chi_\Omega + n \cdot \widetilde{\mathsf T}(\uvels,\upress) \delta_{\Sigma} 
 + (\zeta \cdot n)  h_0 \delta_{\Sigma}}
  \\
& \qquad
 + P \ast_{\R^3}\nabla \cdot \nb{ (n \otimes h_0 )\delta_{\Sigma}  }
 \\
& \qquad + P \ast_{\R^3} \nb{\zeta (h_0\cdot n) \delta_\Sigma},
\end{aligned}
\label{eq:repr.pres.lin.s}
\\
&\begin{aligned}
\upresp
&= c_\perp + Q \ast_G \nb{ f_\perp \chi_\Omega + n \cdot \widetilde{\mathsf T}(\uvelp,\upresp) \delta_{\Sigma} 
 + (\zeta \cdot n)  h_\perp )\delta_{\Sigma}}
  \\
& \qquad
 + Q \ast_G \nabla \cdot \nb{(n \otimes h_\perp )\delta_{\Sigma} } 
 \\
& \qquad + Q \ast_G \nb{\zeta (h_\perp\cdot n) \delta_\Sigma}
+ (\delta_\torus \otimes E ) \ast_G \nb{(\partial_t h \cdot n)\delta_\Sigma} 
\end{aligned}
\label{eq:repr.pres.lin.p}
\end{align}
where $c(t)=c_0+c_\perp(t)$ is a function only depending on $t$. 

We next prepare several estimates of the convolutions appearing in \eqref{eq:repr.vel.lin.s}--\eqref{eq:repr.pres.lin.p}. We define the Euclidean ball of radius $R>0$ by $B_R=\{ x \in {\mathbb R}^3 : |x| <R \}$,  along with the exterior domain $B^R=\{ x \in {\mathbb R}^3 : |x| > R \}$, and the spherical shell $B_{R_1,R_2} =\{ x \in {\mathbb R}^3 : R_1 < |x| < R_2 \}$.

Firstly, we consider the terms with contributions at the boundary.

\begin{lem}\label{lem:conv.bdry}
Let $S>0$ such that
$\Sigma\subset B_{S}$.
Let $\zeta\in\R^3$ such that $0<\snorm{\zeta}\leq\zeta_0$ 
for some $\zeta_0>0$.
Then there is $C=C(\Sigma,S,\zeta_0,\per)>0$ such that
for all 
$\psi=\psi_1\delta_\Sigma$
with
$\psi_1\in\LS{1}(\torus\times\Sigma)$,
and for $\snorm{x}\geq S$
it holds
\begin{equation}
\begin{aligned}
&\nu^1_1(x;\zeta)\, \snorml{\snorm{\fsolss\otimes 1_\torus}\ast \psi(t,x)}
+ \nu^{3/2}_{3/2}(x;\zeta)\, \snorml{\snorm{\nabla\fsolss\otimes 1_\torus}\ast \psi(t,x)}
\\
&\qquad
+ \nu^{2}_{2}(x;\zeta)\,
\snorml{\snorm{\nabla^2\fsolss\otimes 1_\torus}\ast \psi(t,x)}
\\
& \qquad 
+ \snorm{x}^3\, \snorml{\snorm{\fsolpp}\ast \psi(t,x)}
+ \snorm{x}^4\, \snorml{\snorm{\nabla\fsolpp}\ast \psi(t,x)}
+ \snorm{x}^5\, \snorml{\snorm{\nabla^2\fsolpp}\ast \psi(t,x)}
\\
&\qquad
+ |x|\, \snorml{ (E\otimes \delta_\torus)\ast \psi(t,x)}
+ |x|^2\, \snorml{\snorm{Q}\ast \psi(t,x)}
+ |x|^3\, \snorml{\snorm{\nabla Q}\ast \psi(t,x)}
\\
&\qquad
\leq C \norm{\psi_1}_{\LS{1}(\torus\times\Omega_{R})}.
\end{aligned}
\label{eq:conv.bdry1}
\end{equation} 
Moreover, if $\int_\Sigma \psi(t,x)\,\dS(x)=0$,
then 
\begin{equation}
\snorm{x}^2 \snorml{ (E\otimes \delta_\torus)\ast \psi(t,x)}
+ |x|^3\, \snorml{\snorm{Q}\ast \psi(t,x)}
+ |x|^4\, \snorml{\snorm{\nabla Q}\ast \psi(t,x)}
\leq C\norm{\psi_1}_{\LS{1}(\torus\times\Sigma)}.
\label{eq:conv.bdry2}
\end{equation}
\end{lem}

\begin{proof}
Let $R\in(0,S)$ such that $\Sigma\subset B_R$.
For $\snorm{x}\geq S>R\geq \snorm{y}$ we have
\[
\begin{aligned}
\snorm{x-y}
&\geq \snorm{x}-\snorm{y}
\geq (1-R/S)\snorm{x}
\geq S-R,
\\
(1+2\zeta_0 R)\np{1+\wakefct(x-y)}
&\geq 1+2\snorm{\zeta}\snorm{y}+\wakefct(x-y)
\geq 1+\wakefct(x).
\end{aligned}
\]
This yields 
$\nu^\alpha_\beta(x;\zeta)\leq C\,\nu^\alpha_\beta(x-y;\zeta)$
for $\alpha,\beta\geq 0$ and a constant $C=C(\alpha,\beta,R,S,\zeta_0)>0$.
Therefore, 
for any function $\Theta$ with $\snorm{\Theta(t,z)}\leq C\nu^{-\alpha}_{-\beta}(z;\zeta)$ for $\snorm{z}\geq S-R$,
we obtain
\[
\begin{aligned}
\snorml{\Theta \ast \psi(t,x)}
&\leq C\int_\torus\int_{\Sigma} \nu^{-\alpha}_{-\beta}(x-y;\zeta)\,\snorm{\psi_1(s,y)}\,\dS(y)\ds
\\
&\leq C 
\nu^{-\alpha}_{-\beta}(x;\zeta)\norm{\psi_1}_{\LS{1}(\torus\times\Sigma)}.
\end{aligned}
\]
In this proof and the ones that follow, $C$ represents a generic positive constant that may take different values in different steps of the argument. Moreover, if $\snorm{\nabla \Theta (t,z)}\leq C\nu^{-\alpha}_{-\beta}(z;\zeta)$
and $\int_\Sigma \psi(t,y)\,\dS(y)=0$,
then we obtain
\[
\begin{aligned}
\snorml{\Theta \ast \psi(t,x)}
&=\snormL{\int_\torus\int_\Sigma 
\bp{\Theta(t-s,x-y)-\Theta(t-s,x)}\psi_1(s,y)\,\dS(y)\ds}
\\
&=\snormL{\int_\torus\int_\Sigma 
\int_0^1
y\cdot \nabla \Theta(t-s,x-\theta y) \psi_1(s,y)\,\dtheta\dS(y)\ds}
\\
&\leq C R\int_\torus\int_{\Sigma} \nu^{-\alpha}_{-\beta}(x-y;\zeta)\,\snorm{\psi_1(s,y)}\,\dS(y)\ds
\\
&\leq C
\nu^{-\alpha}_{-\beta}(x;\zeta)\norm{\psi_1}_{\LS{1}(\torus\times\Sigma)}.
\end{aligned}
\]
Due to the estimates~\eqref{est:fsolss.decay},~\eqref{est:fsolpp.decay} 
and the decay properties of $E$, $Q$ and $\nabla Q$,
the claim follows from this general result.
\end{proof}

We now consider convolutions of the fundamental solution 
with
functions with suitable spatial decay. 
Since we assume different decay estimates 
of the steady-state and the purely periodic part,
we study them separately.
For the steady-state part, we have the following result.

\begin{lem}
\label{lem:fsolss.conv}
There is $C>0$ such that for all $\zeta\in\R^3$ with $0<\snorm{\zeta}\leq\zeta_0$
for some $\zeta_0>0$,
for all
$g\in\LS{6/5}(\R^3)$ with $\nu^{5/2}_1(\cdot;\zeta)\,g\in\LS{\infty}(\R^3)$, 
and for all $x\in\R^3\setminus\set{0}$
it holds
\[
\begin{aligned}
\nu^1_1(x;\zeta)\, 
\snorml{\snorm{\fsolss}\ast g(x)}
&+ \nu^{3/2}_{3/2}(x;\zeta)\, \snorml{\snorm{\nabla\fsolss}\ast g(x)}
\\
&\quad
+\snorm{x}^2\, \snorml{\snorm{\fsolpres}\ast g(x)}
\leq C \norm{\nu^{5/2}_1(\cdot;\zeta)\, g}_{\LS{\infty}(\R^3)}.
\end{aligned}
\]
\end{lem}

\begin{proof}
This follows from~\cite[Theorem 4.7]{DeuringKracmarExtStatNSFApprBddDom2004}.
\end{proof}

For the purely periodic part, we have the following estimates.

\begin{lem}
\label{lem:fsolpp.conv}
Let $\varepsilon>0$, $r\in[1,\infty)$ and $ \mu >  3$. 
Then there is $C>0$ such that for all $\zeta\in\R^3\setminus\set{0}$ and
$g\in\LSloc{1}(\torus\times\R^3)$ with
$\np{1+\nu^\mu} g \in\LS{r}(\torus;\LS{\infty}(\R^3)^3)$, 
and for all $x\in\R^3$ with $\snorm{x}\geq\varepsilon$
it holds
\[
\begin{aligned}
&\snorm{x}^{3}\,
\snorml{\snorm{\fsolpp}\ast g(t,x)}
+ \snorm{x}^{\min\set{\mu,4}}\, \snorml{\snorm{\nabla\fsolpp}\ast g(t,x)}
\\
&\qquad\qquad 
+ \snorm{x}^2\, \snorml{\snorm{Q}\ast g(t,x)}
\leq C \norm{(1+\nu^\mu)g}_{\LS{r}(\torus;\LS{\infty}(\R^3))}.
\end{aligned}
\]
\end{lem}

\begin{proof}
Set $M\coloneqq \norm{(1+\nu^\mu)g}_{\LS{r}(\torus;\LS{\infty}(\R^3))}$.
We start with the estimate of $\snorm{\fsolpp}\ast g$.
We use H\"older's inequality on $\torus$ and 
Minkowski's integral inequality
and
split the spatial integral into three parts 
to obtain
\[
\begin{aligned}
\snorml{\snorm{\fsolpp}\ast g(t,x)}
&\leq\int_{\R^3} \Bp{\int_\torus\snorml{\fsolpp(t-s,x-y)}^{r'}\,\ds}^{1/r'}
\Bp{\int_\torus\snorml{g(t,y)}^r\,\ds}^{1/r}\dy
\\
&\leq C M \sum_{j=1}^3 \int_{A_j} \Bp{\int_\torus\snorml{\fsolpp(s,x-y)}^{r'}\,\ds}^{1/r'}
\bp{1+\snorm{y}^{\mu} }^{-1}  \,\dy
\eqqcolon C M \sum_{j=1}^3 I_j
\end{aligned}
\]
where $r'=r/(r-1)$,
and we set $A_1=B_R$, $A_2=B^{4R}$ and $A_3=B_{R,4R}$ with $R=\snorm{x}/2$.
First, since $\snorm{y}\leq R$ implies $\snorm{x-y}\geq \snorm{x}/2$,
we can use \eqref{est:fsolpp.decay} to obtain
\[
\begin{aligned}
I_1
\leq C\int_{B_R} \snorm{x-y}^{-3}\np{1+\snorm{y}^{\mu}}^{-1}\,\dy 
&\leq C\snorm{x}^{-3}\int_{B_R} \np{1+\snorm{y}}^{-\mu}\,\dy 
\leq  C\snorm{x}^{-3}
\end{aligned}
\]
since $\mu>3$.
For the second integral,
we again use \eqref{est:fsolpp.decay} 
and that $\snorm{y}\geq 4R$ implies $\snorm{x-y}\geq \snorm{y}/2$
to obtain
\[
\begin{aligned}
I_2
&\leq C\int_{B^{4R}}\snorm{x-y}^{-3}\np{1+\snorm{y}}^{-\mu}\,\dy 
\leq C\int_{B^{4R}}\snorm{y}^{-3} \snorm{y}^{-\mu}\,\dy
=C \snorm{x}^{-\mu}.
\end{aligned}
\]
For the third integral, we note that $r'>1$ and  $\mu-3 >0$, so that we can choose $\tilde r\in(1,5/3)$ such that $\tilde r<r'$
and $3/\tilde{r}' <\mu-3$.
Then H\"older's inequality and \eqref{eq:fsolpp.intfct}
yield
\[
\begin{aligned}
I_3
&\leq C \snorm{x}^{-\mu}\Bp{\int_{B_{R,4R}}1 \dy}^{1/\tilde r'}
\Bp{\int_{\torus}\int_{B_{R,4R}} \snorml{\fsolpp(s,y)}^{\tilde r}\,\dy\ds}^{1/\tilde r}
\\
& = C \snorm{x}^{-\mu} R^{3/\tilde r'}
\Bp{\int_{\torus}\int_{B_{R,4R}} \snorml{\fsolpp(s,y)}^{\tilde r}\,\dy\ds}^{1/\tilde r}
\\
& \leq C\snorm{x}^{-\mu+3/\tilde r'}
\leq C\snorm{x}^{-3}. 
\end{aligned}
\]
Collecting the estimates of $I_1$, $I_2$ and $I_3$,
we arrive at 
\[
\snorml{\snorm{\fsolpp}\ast g(t,x)}
\leq C\snorm{x}^{-3}
\]
as asserted. 
For the estimate of $\snorm{\nabla\fsolpp}\ast g$
we proceed similarly. 
At first, we obtain
\[
\begin{aligned}
\snorml{\snorm{\nabla\fsolpp}\ast g(t,x)}
\leq C M \sum_{j=1}^3 \int_{A_j}\! \Bp{\int_\torus\snorml{\nabla\fsolpp(s,x-y)}^{r'}\,\ds}^{\sfrac{1}{r'}}
\!\!\bp{1+\snorm{y}}^{-\mu} \,\dy
\eqqcolon C M\sum_{j=1}^3 J_j
\end{aligned}
\]
for the sets $A_j$, $j=1,2,3$, as before.
Repeating the above arguments, we can estimate $J_1$ and $J_2$ as
\[
\begin{aligned}
J_1 
\leq C\int_{B_R} \snorm{x-y}^{-4}\np{1+\snorm{y}}^{-\mu}\,\dy 
&\leq C\snorm{x}^{-4}\int_{B_R} \np{1+\snorm{y}}^{-\mu}\,\dy
\leq C\snorm{x}^{-4}, 
\\
J_2
\leq C\int_{B^{4R}}\snorm{x-y}^{-4}\np{1+\snorm{y}}^{-\mu}\,\dy 
&\leq C\int_{B^{4R}}\snorm{y}^{-4} \snorm{y}^{-\mu}\,\dy
=C \snorm{x}^{-1-\mu},
\end{aligned}
\]
and for $J_3$ 
we use $\nabla\fsolpp\in\LS{1}(\torus\times\R^3)$ by~\eqref{eq:fsolpp.intgrad}
to deduce
\[
\begin{aligned}
J_3
&\leq C\np{1+\snorm{x}}^{-\mu}
\int_{\torus}\int_{B_{R,4R}} \snorml{\nabla\fsolpp(s,y)}\,\dy\ds
\leq C\snorm{x}^{-\mu}.
\end{aligned}
\]
In total, these estimates yield
\[
\snorml{\snorm{\nabla\fsolpp}\ast g(t,x)}
\leq C\snorm{x}^{-\min\set{4,\mu}} .
\]
For the convolutions with $Q=\fsolpres\otimes\delta_\torus$,
we use 
$\snorm{\fsolpres(x)}=C\snorm{x}^{-2}$
and argue similarly.
\end{proof}

We now combine the derived pointwise estimates
with the results on time-periodic maximal regularity
established in~\cite{EiterShibata_FlowPastBodyOscBdry}.
This leads to existence of solutions with suitable spatial decay.

\begin{thm}\label{thm:linear.existence+decay}
Let $\Omega\subset\R^3$ be an exterior domain with $\CS{2}$-boundary,
and let  
$\zeta_0>0$ and $p,q\in(1,\infty)$.
Let $h\in \tracesptp$,
and let $f\in\LSloc{1}(\torus\times\Omega)^3$
such that $f=f_0+f_\perp$ satisfies
$\nu^{5/2}_1(\cdot;\zeta)\,f_0\in\LS{\infty}(\Omega)^3$ 
and
$\nu^{3+\delta} f_\perp\in\LS{p}(\torus;\LS{\infty}(\Omega)^3)$
for some $\delta>0$.
For any $\zeta\in\R^3\setminus\set{0}$
there exists a unique solution $(\uvel,\upres)$ to \eqref{eq:oseentp}
satisfying
\begin{equation}
\uvel\in\LS{p}(\torus;\DS{2}{q}(\Omega)^3),
\qquad
\partial_t\uvel\in\LS{p}(\torus;\LS{q}(\Omega)^3),
\qquad
\upres\in\LS{p}(\torus;\DS{1}{q}(\Omega)),
\label{eq:linear.solreg}
\end{equation}
and the estimates
\begin{align}
\begin{split}
\norm{\partial_t\uvel}_{\LS{p}(\torus; \LS{q}(\Omega))}
&+ \norm{\nabla^2\uvel}_{\LS{p}(\torus; \LS{q}(\Omega))}
+ \norm{\nabla\upres}_{\LS{p}(\torus; \LS{q}(\Omega))}
\\
&\qquad\qquad
\leq C\bp{\norm{f}_{\LS{p}(\torus; \LS{q}(\Omega))}
+\norm{h}_{\tracesptp}},
\end{split}
\label{eq:linest.maxreg}
\\
\begin{split}
\nu^1_1(x;\zeta)\, 
\snorml{\uvels(x)}
&+ \nu^{3/2}_{3/2}(x;\zeta)\, \snorml{\nabla\uvels(x)}
+\snorm{x}^2\, \snorml{\upress(x)}
\\
&\qquad\qquad
\leq C \bp{\norm{\nu^{5/2}_1(\cdot;\zeta)\, f_0}_{\LS{\infty}(\Omega)}
+\norm{h_0}_{\tracespss}},
\end{split}
\label{eq:linest.pwss}
\\
\begin{split}
\snorm{x}^{2}
\snorml{\uvelp(t,x)}
&+ \snorm{x}^{3} \snorml{\nabla\uvelp(t,x)}
+ \snorm{x}\, \snorml{\upresp(t,x)}
\\
&\qquad\qquad
\leq C \bp{
\norm{\nu^{3+\delta} f_\perp}_{\LS{p}(\torus;\LS{\infty}(\Omega))}
+\norm{h_\perp}_{\tracesptp}}.
\end{split}
\label{eq:linest.pwpp}
\end{align}
If the total flux through $\Sigma$ is constant, that is, if \eqref{eq:constantflux} holds,
then \eqref{eq:linest.pwpp}
can be replaced with
\begin{equation}
\begin{split}
\snorm{x}^{3}
\snorml{\uvelp(t,x)}
&+ \snorm{x}^{\min\set{3+\delta,4}} \snorml{\nabla\uvelp(t,x)}
+ \snorm{x}^2\, \snorml{\upresp(t,x)}
\\
&\qquad\qquad
\leq C \bp{
\norm{\nu^{3+\delta} f_\perp}_{\LS{p}(\torus;\LS{\infty}(\Omega))}
+\norm{h_\perp}_{\tracesptp}}.
\end{split}
\label{eq:linest.pwpp.constflux}
\end{equation}
Here $C=C(\Omega, p,q, \delta,\zeta_0)>0$ if $\snorm{\zeta}\leq\zeta_0$.
\end{thm}
\begin{proof}
We first show that $f\in\LS{p}(\torus;\LS{s}(\Omega))$
for all $s\in(1,\infty)$. 
With the integral $\mathcal J_R(a,b)$ from~\eqref{eq:Jab}
and the estimate~\eqref{far}
we obtain
\[
\begin{aligned}
\int_{B^{R}}\!\snorm{f_0(x)}^s\,\dx
&\leq \norm{f_0}_{\infty,\nu^{5/2}_1(\cdot;\zeta);\Omega}^s \int_{R}^\infty \! \mathcal J_r(\tfrac{5s}{2},s)\,\dr 
\leq C \norm{f_0}_{\infty,\nu^{5/2}_1(\cdot;\zeta);\Omega}^s\int_{R}^\infty r^{-5s/2+1}\,\dr.
\end{aligned}
\]
Moreover, 
we have
\[
\begin{aligned}
\int_\torus\Bp{\int_{\Omega}\snorm{f_\perp(t,x)}^s\,\dx}^{p/s}\dt
&\leq
\norm{f_\perp}_{\infty, \nu^{3+\delta};\torus\times\Omega}^p
\Bp{\int_{\Omega}|x|^{-(3+\delta)s}\,\dx}^{p/s}.
\end{aligned}
\]
Since the remaining integrals in both estimates are finite,
we obtain $f\in\LS{p}(\torus;\LS{s}(\Omega)^3)$ for any $s\in(1,\infty)$.
Therefore, the existence of a solution $(\uvel,\upres)$ 
in the class given by~\eqref{eq:linear.solreg}
and subject to inequality~\eqref{eq:linest.maxreg}
follows from~\cite[Theorem~4.7]{EiterShibata_FlowPastBodyOscBdry}. Since we can choose any $s<2$,
the velocity field solution satisfies $\uvel\in\LS{p}(\torus;\LS{q}(\Omega)^3)$
for $q\in(2,\infty)$ 
and is unique. Moreover, the pressure field is unique up to addition by a function constant in space,
which corresponds to the function $c=c_0+c_\perp$ in the representation formulas~\eqref{eq:repr.pres.lin.s} and~\eqref{eq:repr.pres.lin.p}
for the pressure.
Fixing $c\equiv 0$, we ensure uniqueness of $\upres$.

To derive the pointwise estimates~\eqref{eq:linest.pwss}, \eqref{eq:linest.pwpp}
and~\eqref{eq:linest.pwpp.constflux},
we use 
the representation formulas~\eqref{eq:repr.vel.lin.s} and~\eqref{eq:repr.vel.lin.p} 
for the steady-state and purely periodic parts 
of the velocity field.
Similarly, we use~\eqref{eq:repr.pres.lin.s} and~\eqref{eq:repr.pres.lin.p}
to obtain the estimates of the pressure $\upres$.
Then the asserted estimates follow directly from
Lemma~\ref{lem:conv.bdry},
Lemma~\ref{lem:fsolss.conv} and Lemma~\ref{lem:fsolpp.conv},
where we use
\[
\begin{aligned}
\norm{n \cdot \widetilde{\mathsf T}(\uvels,\upress)}_{\LS{1}(\Sigma)}
&\leq C\norm{\widetilde{\mathsf T}(\uvels,\upress)}_{\WS{1}{q}(\Omega_{R})}
\\
&\leq C\bp{\norm{\nabla^2\uvels}_{\LS{q}(\Omega)}
+\norm{\nabla\upress}_{\LS{q}(\Omega)}
+\norm{h_0}_{\LS{q}(\Omega)}}
\\
&\leq C\bp{\norm{f_0}_{ \LS{q}(\Omega)}
+\norm{h_0}_{\tracespss}},
\\
\norm{n \cdot \widetilde{\mathsf T}(\uvelp,\upresp)}_{\LS{1}(\torus\times\Sigma)}
&\leq C\norm{\widetilde{\mathsf T}(\uvelp,\upresp)}_{\LS{p}(\torus;\WS{1}{q}(\Omega_{R}))}
\\
&\leq C\bp{\norm{\nabla^2\uvelp}_{\LS{p}(\torus;\LS{q}(\Omega))}
+\norm{\nabla\upresp}_{\LS{p}(\torus;\LS{q}(\Omega))}
 +\norm{h_\perp}_{\tracesptp}}
\\
&\leq C\bp{\norm{f_\perp}_{\LS{p}(\torus; \LS{q}(\Omega))}
+\norm{h_\perp}_{\tracesptp}}
\end{aligned}
\]
due to~\eqref{eq:linest.maxreg},
 where we choose any $R>0$ such that $\partial\Omega\subset B_R$. 
Observe that in the general case, 
the pointwise asymptotic behavior of $\uvelp$, $\nabla\uvelp$ and $\upresp$ is determined by the latter term in the representation formulas~\eqref{eq:repr.vel.lin.p}
and~\eqref{eq:repr.pres.lin.p},
which leads to estimate~\eqref{eq:linest.pwpp} by using estimate~\eqref{eq:conv.bdry1} from Lemma~\ref{lem:conv.bdry}.
If we assume~\eqref{eq:constantflux},
we also have 
\[
\int_\Sigma h_\perp(t,x)\cdot n\,\dS(x)
=0,
\]
so that those terms can be estimated with~\eqref{eq:conv.bdry2} from Lemma~\ref{lem:conv.bdry} instead,
which leads to the better decay rate stated in~\eqref{eq:linest.pwpp.constflux}.
\end{proof}

\subsection{Solutions to the nonlinear problem}

For $k=0,1$ and $\delta\in(0,1]$, we introduce the function space
\[
\begin{aligned}
\calX_k
&\coloneqq\setcl{\vvel\in\LS{p}(\torus;\WSloc{2}{q}(\Omega)^3)\cap\WS{1}{p}(\torus;\LS{q}(\Omega)^3)}{
\Div\vvel=0,\,\norm{\vvel}_{\calX_k}<\infty},
\\
\norm{\vvel}_{\calX_k}
&\coloneqq
\norm{\nabla^2\vvel}_{\LS{p}(\torus;\LS{q}(\Omega))}
+\norm{\partial_t\vvel}_{\LS{p}(\torus;\LS{q}(\Omega))}
\\
&\qquad
+\norm{\proj\vvel}_{\infty,\nu^1_1(\cdot;\zeta); \Omega}
+\norm{\nabla\proj\vvel}_{\infty,\nu^{3/2}_{3/2}(\cdot;\zeta); \Omega}
+N_k(\projcompl\vvel)
\end{aligned}
\]
where
\[
\begin{aligned}
N_0(\wvel)
&\coloneqq
\norm{\wvel}_{\infty,\nu^{2}; \torus\times \Omega}
+\norm{\nabla\wvel}_{\infty,\nu^{3}; \torus\times \Omega},
\\
N_1(\wvel)
&\coloneqq
\norm{\wvel}_{\infty,\nu^{3}; \torus\times \Omega}
+\norm{\nabla\wvel}_{\infty,\nu^{3+\delta}; \torus\times \Omega}.
\end{aligned}
\]
For given $\vvel\in\calX_k$, 
we consider the problem
\begin{equation}\label{eq:nstp.lin.fp}
\left\{\ \begin{aligned}
\partial_t\uvel-\Delta\uvel-\zeta\cdot\nabla\uvel+\nabla\upres&=f
-\calN(\vvel,\vvel)
&&
\tin\torus\times\Omega, \\
\nabla \cdot \uvel&=0 
&&
\tin\torus\times\Omega, \\
\uvel&=h
&&
\ton\torus\times\Sigma,
\end{aligned}\right.
\end{equation}
where the nonlinear term $\calN$ is defined as
\[
\calN(\vvel_1,\vvel_2)\coloneqq\vvel_1\cdot\nabla\vvel_2.
\]
Below we show that the linear theory from 
Theorem~\ref{thm:linear.existence+decay}
provides a solution $(\uvel,\upres)$
to this problem if $\vvel\in\calX_k$.
This defines a solution map
$\calS\colon \vvel\mapsto\uvel$,
and $(\uvel,\upres)$ solves the nonlinear problem~\eqref{eq:nstp}
if $\uvel$ is a fixed point of $\calS$.
For obtaining such a fixed point, 
we first 
prove the following estimates of the convection term,
where we again distinguish steady-state and purely periodic part.

\begin{lem}
\label{lem:nonl.est}
Let $k\in\set{0,1}$ and let $\vvel_1,\vvel_2\in\calX_k$.
Then
\[
\begin{aligned}
&\norm{\nu^{5/2}_1(\cdot;\zeta)\, \proj\calN(\vvel_1,\vvel_2)}_{\LS{\infty}(\Omega)}
+\norm{\nu^{7/2+k/2}\,\projcompl\calN(\vvel_1,\vvel_2)}_{\LS{\infty}(\torus\times\Omega)}
\leq C\norm{\vvel_1}_{\calX_k}\norm{\vvel_2}_{\calX_k}.
\end{aligned}
\]
\end{lem}

\begin{proof}
We set $\vvel_j=\zvel_j+\wvel_j$ with $\zvel_j\coloneqq\proj\vvel_j$ 
and $\wvel_j=\projcompl\vvel_j$ for $j=1,2$.
Then we have 
\[
\begin{aligned}
\proj\calN(\vvel_1,\vvel_2) 
&= \zvel_1\cdot\nabla\zvel_2 + \proj(\wvel_1\cdot\nabla\wvel_2),
\\
\projcompl\calN(\vvel_1,\vvel_2) 
&= \zvel_1\cdot\nabla\wvel_2 
+ \wvel_1\cdot\nabla\zvel_2
+ \projcompl(\wvel_1\cdot\nabla\wvel_2).
\end{aligned}
\]
Therefore, for $x\in\Omega$ we can estimate
\[
\begin{aligned}
\nu^{5/2}_1(x;\zeta)\, &\snorm{\proj\calN(\vvel_1,\vvel_2)(x)}
\\
&\leq C \left( \nu^{1}_1(x;\zeta)\, \snorm{\zvel_1(x)}\, 
\snorm{x}^{3/2}\snorm{\nabla\zvel_2(x)}
+ \nu^{0}_1(x;\zeta)\, \snorm{\wvel_1(t,x)}  \snorm{x}^{5/2} \snorm{\nabla\wvel_2(t,x)} \right)
\\
&\leq C \norm{\vvel_1}_{\calX_k}\norm{\vvel_2}_{\calX_k},
\end{aligned}
\]
and
\[
\begin{aligned}
\snorm{x}^{7/2+k/2}\, &\snorm{\projcompl\calN(\vvel_1,\vvel_2)(t,x)}
\\
&\leq C\bp{\snorm{x}\, \snorm{\zvel_1(x)}\,\snorm{x}^{5/2+k/2}\snorm{\nabla\wvel_2(t,x)}
+ \snorm{x}^{2+k/2}\, \snorm{\wvel_1(t,x)}\,\snorm{x}^{3/2}\,\snorm{\nabla\zvel_2(x)}
\\
&\qquad\qquad
+\snorm{x}^{2+k/2}\, \snorm{\wvel_1(t,x)}\,\snorm{x}^{3/2}\,\snorm{\nabla\wvel_2(t,x)}
}
\\
&\leq C \norm{\vvel_1}_{\calX_k}\norm{\vvel_2}_{\calX_k}.
\end{aligned}
\]
This shows the asserted estimates.
\end{proof}

We can now show existence of a solution to~\eqref{eq:nstp}
by a fixed-point argument.

\begin{proof}[Proof of Theorem~\ref{thm:strongsol}]
We set $k=0$ in the general case and we set $k=1$ when \eqref{eq:constantflux} is satisfied.
For $\varepsilon>0$
consider the set 
\[
\calX_{k,\varepsilon}\coloneqq
\setcl{\vvel\in \calX_k}{\norm{\vvel}_{\calX_k}\leq\varepsilon}.
\]
In virtue of Lemma~\ref{lem:nonl.est}
and Theorem~\ref{thm:linear.existence+decay},
for any $\vvel\in \calX_k$ 
there exists a solution $(\uvel,\upres)$ to~\eqref{eq:nstp.lin.fp}
with the regularity stated in~\eqref{eq:linear.solreg}
and subject to the estimates
\[
\begin{aligned}
&\nu^1_1(x;\zeta)\, 
\snorml{\uvels(x)}
+ \nu^{3/2}_{3/2}(x;\zeta)\, \snorml{\nabla\uvels(x)}
+ \snorm{x}^2\, \snorml{\upress(x)}
\\
&\qquad
\leq C \bp{
\norm{\nu^{5/2}_1(\cdot;\zeta)\, \proj\calN(\vvel,\vvel)}_{\LS{\infty}(\Omega)}
+\norm{\nu^{5/2}_1(\cdot;\zeta) f_0}_{\LS{\infty}(\Omega)}
+\norm{h_0}_{\tracespss}}
\\
&\qquad
\leq C\bp{\norm{\vvel}_{\calX_k}^2+\varepsilon^2},
\\
&\snorm{x}^{2+k}
\snorml{\uvelp(t,x)}
+ \snorm{x}^{3+\min\set{k,\delta}} \snorml{\uvelp(t,x)}
+ \snorm{x}^{1+k}\, \snorml{\upresp(t,x)}
\\
&\qquad
\leq C \bp{
\norm{\nu^{3+\delta}\projcompl\calN(\vvel,\vvel)}_{\LS{p}(\torus;\LS{\infty}(\Omega))}
+\norm{\nu^{3+\delta} f_\perp}_{\LS{p}(\torus;\LS{\infty}(\Omega))}
+\norm{h_\perp}_{\tracesptp}}
\\
&\qquad
\leq C \bp{\norm{\vvel}_{\calX_k}^2+\varepsilon^2}.
\end{aligned}
\]
For $\vvel\in \calX_{k,\varepsilon}$, we thus have
\[
\norm{\uvel}_{\calX_k}\leq C \varepsilon^2 \leq \varepsilon
\]
if $\varepsilon>0$ is chosen sufficiently small.
Then the solution map $\calS\colon\vvel\mapsto\uvel$ is a well-defined 
self mapping $\calS\colon {\calX}_{k,\varepsilon}\to {\calX}_{k,\varepsilon}$.
Moreover, for $\vvel_1,\vvel_2\in\calX_{k,\varepsilon}$, 
the differences $\ouvel=\uvel_1-\uvel_2$ and $\oupres=\upres_1-\upres_2$, 
where $\uvel_j\coloneqq\calS(\vvel_j)$ with corresponding pressure $\upres_j$, $j=1,2$, satisfy
\[
\left\{ \ 
\begin{aligned}
\partial_t\ouvel-\Delta\ouvel-\zeta\cdot\nabla\ouvel+\nabla\oupres&=
-\calN(\vvel_1,\vvel_1)+\calN(\vvel_2,\vvel_2)
&&
\tin\torus\times\Omega, \\
\nabla \cdot \ouvel&=0 
&&
\tin\torus\times\Omega, \\
\ouvel&=0
&&
\ton\torus\times\Sigma.
\end{aligned}
\right.
\]
Noting that 
\[
\calN(\vvel_1,\vvel_1)-\calN(\vvel_2,\vvel_2)
=\calN(\vvel_1-\vvel_2,\vvel_1)+\calN(\vvel_2,\vvel_1-\vvel_2),
\]
we can adapt the same argument as before to conclude the estimate
\[
\norm{\calS(\vvel_1)-\calS(\vvel_2)}_{\calX_k}
=\norm{\overline\uvel}_{\calX_k}
\leq C\np{\norm{\vvel_1}_{\calX}+\norm{\vvel_2}_{\calX_k}}
\norm{\vvel_1-\vvel_2}_{\calX_k}
\leq 2C\varepsilon\norm{\vvel_1-\vvel_2}_{\calX_k}.
\]
Hence, choosing $\varepsilon>0$ sufficiently small,
we obtain that $\calS$ is also a contraction.
Finally, the contraction mapping principle yields
the existence of a unique fixed point $\uvel=\calS(\uvel)\in\calX_{k,\varepsilon}$.
If $\upres$ denotes the associated pressure, 
then $(\uvel,\upres)$ is a solution to \eqref{eq:nstp}
with the asserted properties.
\end{proof}

\section{Existence in the truncated domains}
\label{truncateddomain}

Our aim is to find a solution $( \vvel,\vpres)$ to the problem \eqref{eq:nstp.pert}--\eqref{abc} defined in the truncated domain $\Omega_R$. We seek a velocity field in the form $
v = \widetilde{h} + \vartheta$ with $\widetilde{h}$ an appropriate extension of $h$ to $\Omega _R$ and $
\vartheta \in  L^2(\mathbb T;H^1(\Omega_R)^3)$ satisfying $\vartheta|_{\Sigma} = 0$ and $\nabla \cdot  \vartheta =0$ in ${\mathbb T} \times \Omega_R$.

\subsection{Functions spaces over the truncated domains}

In what follows, the usual inner products in $L^2(\Omega_R)$ and $L^2(\partial B_R)$ will be denoted by $(\cdot, \cdot )_{\Omega_R}$ and $(\cdot,\cdot)_{\partial B_R}$, respectively. As in \cite{DeuringKracmarExtStatNSFApprBddDom2004,DKN2021}, we consider $H^1(\Omega_R)$ endowed with inner product and norm
\begin{equation}
( v,w)_{(R)} := (\nabla v,\nabla w)_{\Omega_R} + \frac{1}{R}(v,w)_{\partial B_R}, \,\,
\| w \|_{(R)} = \left( \| \nabla w \|^2_{2,\Omega_R} + \frac{1}{R} \| w \|^2_{2,\partial B_R} \right)^{\sfrac{1}{2}}
\label{ipnR}
\end{equation}
and we equip the space of time-periodic functions $L^2({\mathbb T};H^1(\Omega_R))$ with the norm
\[
\| w \|_{({\mathbb T},R)}:= \left( \int_{\mathbb T} \| w \|_{(R)}^2 \dt \right)^{\sfrac{1}{2}} .
\]

Within this framework, the following estimate holds for time-periodic functions:
\begin{lem}
\label{estdk}
Take a fixed $S \in (0,\infty)$ with $\partial\Omega \subset B_{S}$ and $R > S$. Then there is a constant $C(S)>0$ such that
$$
\left(  \int_{{\mathbb T} \times B_R \setminus \overline{B_{S}}}\frac{|u(t,x)|^2}{|x|^2} \dx \dt \right)^{\sfrac{1}{2}} \leq C(S) \| u \|_{({\mathbb T},R)}
$$
for all $u \in L^2({\mathbb T};H^1(\Omega_R))$.
\end{lem}

\begin{proof}
We can directly apply the reasoning from \cite[Theorem 3.6]{DeuringKracmarExtStatNSFApprBddDom2004}.
\end{proof}

The space 
\[
W_R:= \left\{ w \in H^1(\Omega_R)^3: \, w|_{\Sigma} =0 \right\},
\]
with inner product and norm \eqref{ipnR}, will be relevant in the analysis of problem \eqref{eq:nstp.pert}--\eqref{abc}.

Consider the following subspaces of divergence-free functions of $W_R$,
\[
{\mathscr V}_R:= \left\{\varphi |_{\Omega_R} : \, \varphi \in C_0^{\infty}(\Omega)^3 \text{ and } \nabla \cdot \varphi =0 \textrm{ in } \Omega \right\},
\]
\[
V_R := \textrm{ the closure  of } {\mathscr V}_R
\textrm{ in } H^{1}(\Omega_R)^3,
\]
and the space
\[
H_R := \textrm{ the closure of }  {\mathscr V}_R 
\textrm{ in }  L^2(\Omega_R)^3.
\]
If $\Omega$ is a domain with a Lipschitz continuous boundary, then  
$$
H_R = \left\{v \in L^2 (\Omega_R)^3 : \nabla \cdot v =0
\textrm{ in }  \Omega_R \textrm{ and } v \cdot n = 0 \textrm{ on } {\Sigma} \right\},
$$ 
where $n$ represents the unit outer normal on $\Sigma$, with $\nabla \cdot v =0$ and $v \cdot n$ interpreted in the weak
sense, and
$$
V_R = \left\{v \in H^1 (\Omega_R)^3 : \nabla \cdot v =0
\textrm{ in } \Omega_R \textrm{ and } v = 0 \textrm{ on }
       \Sigma \right\}. 
$$  
For $\sfrac{6}{5} \leq q_1 < 6$ and $ \sfrac{4}{3} \leq q_2 < 4$, we have the embeddings
\begin{equation}
V_R \xhookrightarrow{c}  L^{q_1}(\Omega) \oplus L^{q_2}(\partial B_R) \xhookrightarrow{} V'_R,
\label{compac}
\end{equation}
which are compact and continuous, respectively.
For $f \in L^2(\Omega_R)^3$, a weak solution to the Stokes problem
\[
\left\{ \ 
\begin{aligned}
- \Delta u +  \nabla \mathscr{p}  &=  f && \textrm{ in  } \Omega_R  \\
\nabla \cdot u &= 0  && \textrm{ in  } \Omega_R,  \\
u &= 0 &&  \textrm{ on  }  \Sigma,  \\
\frac{x}{R} \cdot \nabla u -  \mathscr{p} \frac{x}{R}  + \frac 1R u &= 0 && \textrm{ on  }  \partial B_R,
\end{aligned}
\right.
\]
is a field $u \in V_R$ such that 
\[
 (\nabla u,\nabla \varphi)_{\Omega_R} + \frac{1}{R}(u,\varphi)_{\partial B_R} = (f,\varphi)_{\Omega_R}  \quad \forall \varphi \in V_R.
\]
Based on this Stokes problem, it is possible to construct a special basis for the spaces $H_R$ and $V_R$. 
\begin{lem} The spectral problem 
$$
(\nabla \Psi,\nabla \varphi)_{\Omega_R} + \frac{1}{R}(\Psi,\varphi)_{\partial B_R} = \lambda (\Psi,\varphi)_{\Omega_R},  \quad  \forall \varphi \in V_R
$$
admits a sequence $\{\Psi_k\}_{k \in \mathbb N} \subset V_R$ of non-zero solutions corresponding to a sequence $\{\lambda_k\}_{k \in \mathbb N}$ of eigenvalues
$$
0 < \lambda_1 \leq \lambda_2 \leq \lambda_3 \leq ...
$$
which satisfies $\lambda_k \to \infty$ as $k \to \infty$.

Moreover, we can choose $\{\Psi_k\}_{k \in \mathbb N}$ in such a way that it forms an
  orthonormal basis of $H_R$ and $\left\{\sfrac{\Psi_k}{\lambda_k^{\sfrac{1}{2}}}\right\}_{k \in \mathbb N}$ is an orthonormal basis of $V_R$. 
\label{basis}
\end{lem}
\begin{proof} Given $\Psi \in H_R$, by Lax-Milgram Theorem, the problem
$$
(\nabla u,\nabla \varphi)_{\Omega_R} + \frac{1}{R}(u,\varphi)_{\partial B_R}  = (\Psi,\varphi)_{\Omega_R},  \quad \forall \varphi \in V_R
$$
has a unique solution $u \in V_R$. 
The solution operator $\calS : H_R \rightarrow  H_R$,  $\Psi \mapsto u$, is compact, self-adjoint and  positive. Hence, $H_R$ admits an orthonormal basis of eigenfunctions $\Psi_k  \in V_R$ of $\calS$ with corresponding eigenvalues $\mu_k$ satisfying $\mu_k > 0,$ for all $k \in {\mathbb N},$ and $\mu_k \rightarrow 0$ as $k \rightarrow \infty.$ Thus, defining $\lambda_k = \sfrac{1}{\mu_k}$, we obtain
\begin{equation}
(\nabla \Psi_k,\nabla \varphi)_{\Omega_R} + \frac{1}{R}(\Psi_k,\varphi)_{\partial B_R} = \lambda_k (\Psi_k,\varphi)_{\Omega_R}  \quad \forall \varphi \in V_R.
\label{vap2}
\end{equation}

Suppose that $v \in V_R$ satisfies $(\Psi_k,v)_{(R)}=(\nabla \Psi_k,\nabla v)_{\Omega_R} + (\Psi_k,v)_{\partial B_R}\big/R = 0$ for all $k \in \mathbb N$. From \eqref{vap2},  it follows that $(\Psi_k,v)_{\Omega_R} = 0$ for all $k \in \mathbb N$, and since $\{\Psi_k\}_{k \in \mathbb N}$ is a basis of  $H_R$, we conclude that $v \equiv 0$. Hence, the linear span of $\{\Psi_k\}_{k \in \mathbb N}$ is dense in $V_R$. From \eqref{vap2}, we further obtain
\[
\begin{aligned} 
\bigg( \frac{\Psi_k}{\lambda_k^{\sfrac{1}{2}}} , \frac{\Psi_j}{\lambda_j^{\sfrac{1}{2}}} \bigg)_{(R)}
&=\bigg( \frac{\nabla \Psi_k}{\lambda_k^{\sfrac{1}{2}}} , \frac{\nabla \Psi_j}{\lambda_j^{\sfrac{1}{2}}} \bigg)_{\Omega_R} + 
\frac{1}{R}\bigg( \frac{\Psi_k}{\lambda_k^{\sfrac{1}{2}}} , \frac{\Psi_j}{\lambda_j^{\sfrac{1}{2}}} \bigg)_{\partial B_R} 
\\
& = \frac{\lambda_k}{\lambda_k^{\sfrac{1}{2}} \lambda_j^{\sfrac{1}{2}}} (\Psi_k,\Psi_j)_{\Omega_R}
 = \frac{\lambda_k}{\lambda_k^{\sfrac{1}{2}} \lambda_j^{\sfrac{1}{2}}} \delta_{kj}= \delta_{kj}, \qquad \forall j,k \in {\mathbb N}.
\end{aligned}
\] 
Therefore, $\left\{\sfrac{\Psi_k}{\lambda_k^{\sfrac{1}{2}}}\right\}_{k \in \mathbb N}$ is an orthonormal basis of $V_R$. 
\end{proof}

\subsection{Weak solutions in the truncated domain}

Assume $\Omega$ is a Lipschitz domain and recall the total flux of $h$ over $\Sigma$, given by $\Phi(t):=\int_{\Sigma}  h(t,x)  \cdot n(x) \dS(x)$. To simplify the presentation, for each fixed $R$, we define
\begin{equation}
c(u,v,w)  := \int_{\Omega_R}  u \cdot \nabla v \cdot w \, \dx -  \frac 12  \int_{\partial B_R} \left( \frac{x}{R} \cdot u \right) (v  \cdot w) \, \dS(x),
\label{defc}
\end{equation}
which is well defined for $u,v,w \in H^1(\Omega_R)^3$ and satisfies
\begin{equation}
c(u,v,v) = 0, \quad \forall  u \in H_R,\, 
v \in V_R.
\label{cigual0} 
\end{equation}
In what follows, $\sigma(x):= \nabla E(x)= - \frac{x}{4\pi |x|^3}$.
 Observe that $\sigma=-P$ for the pressure part $P$ of fundamental solution, defined in~\eqref{PE}. 
\begin{lem}
\label{div}
Given $h \in H^1({\mathbb T};H^{\sfrac{1}{2}}(\partial \Omega)^3)$
define $\Phi$ as in~\eqref{eq:flux}.
Let $R_0$ be such that $\partial\Omega\subset B_{R_0}$.
For any $\gamma >0$, there exists $\widetilde{h} \in H^1({\mathbb T} ; H^1(\Omega)^3)$ satisfying 
\begin{equation} 
\left\{\ 
 \begin{aligned}
  \nabla \cdot \widetilde{h}  &=  0 &&\text{ in } {\mathbb T} \times \Omega, \\
   \widetilde{h}  &= h   &&\text{ on }   {\mathbb T} \times \Sigma,
    \end{aligned}
    \right.
    \label{tdiv}
\end{equation}
and the estimate
\begin{equation}
|  c(\vartheta,\vartheta,\tilde{h}) |    \leq  \gamma  \| \vartheta  \|_{(R)}^2  + \|  \Phi  \|_{\infty,{\mathbb T}} \left( C_S \| \sigma \|_{3,\Omega_R}  +  \frac{1}{8 \pi  R } \right)  \| \vartheta  \|_{(R)}^2 \text{ in } {\mathbb T}
\label{estdiv}
\end{equation}
for all $R>R_0$,
where $C_S$ is a Sobolev embedding constant.
\end{lem}
\begin{proof} 
Decompose 
\[
h(t,x) = [ h(t,x) -  \Phi(t) \sigma|_{\Sigma}(x) ] + \Phi(t) \sigma|_{\Sigma}(x)  =:  h^{(1)}(t,x) +  h^{(2)}(t,x) , \quad (t,x) \in   {\mathbb T} \times \Sigma.
\]
Then $\int_{\Sigma} h^{(1)}(t,x)  \cdot n(x) \dS = 0$ for all $t \in {\mathbb T}$ and $\nabla \cdot \sigma =0$ in $\Omega$. 
 
For fixed $R_0 >0$ such that $\partial\Omega\subset B_{R_0}$,
we can find (see \cite[Lemma IX.4.1]{GaldiBook2011} and \cite[Lemma 3.3]{TOkabe2011}) $w:{\mathbb T} \times \Omega_{R_0} \to {\mathbb R}^3$ such that
 $$
 \left\{
 \begin{aligned}
 \nabla \times w &= h^{(1)} &&\text{ on }  {\mathbb T} \times \Sigma, \\
  \nabla \times w &= 0 &&\text{ on }  {\mathbb T} \times \partial B_{R_0},\\
  w &= 0 &&\text{ on }  {\mathbb T} \times \partial B_{R_0}, 
 \end{aligned}
 \right.
$$
and 
$$
\| w(t,\cdot) \|_{2,2,\Omega_{R_0}} \leq C(\Omega_{R_0}) \|  h^{(1)}(t,\cdot)\|_{\sfrac{1}{2}, 2 ,\Sigma}, \quad t \in {\mathbb T},
$$
so that $w \in H^1({\mathbb T};H^2(\Omega_{R_0})^3)$ along with the estimate 
$$
\| w \|_{H^1({\mathbb T};H^2(\Omega_{R_0})^3)} \leq C(\Omega_{R_0}) \| h^{(1)}(t,\cdot)\|_{H^1({\mathbb T};H^{\sfrac{1}{2}}(\Sigma)^3)}.
$$
Let $0 < \varepsilon <1$ and $\Psi_\varepsilon \in C^\infty(\R;\R)$
be such that $\Psi_\varepsilon(\theta) = 1$ for $\theta <  \frac{\exp(-2/\varepsilon)}{2}$, $\Psi_\varepsilon(\theta) = 0$ for $\theta \geq 2 \exp(-1/\varepsilon)$, $|\Psi_\varepsilon(\theta)| \leq 1$ and $|\Psi'_\varepsilon(\theta)| \leq \varepsilon/ \theta $, for all $\theta >0$. Define $d(x)$ as the distance of a point $x \in \Omega_{R_0}$ to the boundary $\partial \Omega_{R_0}$ and let $\rho(x)$ be the corresponding regularized distance (in the sense of Stein). Using these, define the cut-off function for the domain $\Omega_{R_0}$ 
$$
\psi_\varepsilon(x) := \Psi_\varepsilon(\rho(x)),
$$ 
and extend it by 1 to the exterior domain $\Omega$. The extension satisfies (see \cite[Lemma III.6.2]{GaldiBook2011} and \cite[Lemma 3.2]{TOkabe2011})
\[
\psi_\varepsilon(x)= \begin{cases}
1 & \text{if} \quad d(x) < \frac{\exp(-2/\varepsilon)}{2 \kappa_1},\\
0 & \text{if} \quad d(x) \geq 2 \exp(-1/\varepsilon),
\end{cases}
\]
and 
\[
\nabla \psi_\varepsilon(x) = \Psi'_\varepsilon(\rho(x)) \nabla \rho(x),   \qquad | \nabla \psi_\varepsilon(x) | \leq \frac{\varepsilon \kappa_2}{d(x)},
\]
where $\kappa_1$ and $\kappa_2$ are positive constants independent of the domain.

Define
\[
\begin{aligned}
 \widetilde{h}(t,x)  & =  \widetilde{h^{(1)} } (t,x)   +  \widetilde{h^{(2)} }(t,x)    = \nabla \times (w(t,x)\psi_\varepsilon(x))  + \Phi(t) \sigma(x) 
 \\
 & = \nabla \psi_\varepsilon(x) \times w(t,x) +  \psi_\varepsilon(x) \nabla \times w(t,x)   + \Phi(t) \sigma(x) , \quad (t,x)  \in {\mathbb T} \times \Omega,
 \end{aligned}
\]
where $w$ is extended to 0 outside ${\mathbb T} \times B_{R_0}$. Clearly, the function $\widetilde{h}$ is divergence free. 

Now assume $R > R_0$. Then, for sufficiently small $\varepsilon$, following
 \cite[Lemma X.4.2]{GaldiBook2011} or \cite[Lemma 3.3]{TOkabe2011}, we can estimate
\[
\begin{aligned} 
 | c(\vartheta,\vartheta,\tilde{h}) |  
&= \left| \int_{\Omega_R} \vartheta \cdot \nabla \vartheta \cdot \tilde{h} \, \dx +  \frac{\Phi(t)}{8 \pi R^2} \| \vartheta \cdot n \|^2_{2, \partial B_R} \right|   \\ 
& \leq  \left|  \int_{\Omega_R} \vartheta \cdot \nabla \vartheta \cdot \nabla \times (w \psi_\varepsilon)  \, \dx \right|  \\
& \quad +  | \Phi(t) | \left| \int_{\Omega_R} \vartheta \cdot \nabla \vartheta  \cdot   \sigma \, \dx \right|  +   \frac{|\Phi(t)|}{8 \pi R^2} \| \vartheta \cdot n \|^2_{2, \partial B_R} \, \\ 
& \leq  \gamma \| \nabla  \vartheta  \|^2_{2,\Omega_R}  + C_S \|  \Phi  \|_{\infty,{\mathbb T}}  \| \sigma \|_{3,\Omega_R}   \| \nabla  \vartheta  \|^2_{2,\Omega_R}  + \frac{\|  \Phi  \|_{\infty,{\mathbb T}} }{8 \pi R^2}  \| \vartheta \|^2_{2, \partial B_R} ,
 \end{aligned}
\]
where $C_S$ is a constant related with the Sobolev embedding in $\Omega$, and $ \vartheta  \in L^2({\mathbb T};V_R)$. 
\end{proof}

Taking into account the regularity of the external force used to solve the exterior problem, we can assume that $f \in L^2({\mathbb T}\times \Omega_R)^3$ in in \eqref{eq:nstp.pert}--\eqref{abc}. Regarding existence and uniqueness of weak solution for \eqref{eq:nstp.pert}--\eqref{abc}, we fix 
\[ 
0 < \gamma < \sfrac{1}{2} - \|  \Phi  \|_{\infty,{\mathbb T}} \left( C_S \| \sigma \|_{3,\Omega_R}  +  \frac{1}{8 \pi  R } \right)
\] 
and a solenoidal extension $\tilde{h} \in H^1({\mathbb T};H^1(\Omega)^3)$ given by Lemma \ref{div}. Then, we will seek the velocity field for system \eqref{eq:nstp.pert}  in the form $v := u_R = \vartheta + \tilde{h}$ where $\vartheta \in  L^2({\mathbb T};V_R)\cap L^\infty({\mathbb T};H_R)$.
The velocity $\vartheta$ and an associated pressure $\vpres$ should satisfy
\[
\left\{
\begin{aligned}
&  \frac{\mathrm d}{\mathrm dt}\int_{\Omega_R}\vartheta\cdot\Psi\,\dx 
+  \int_{\Omega_R} \nabla \vartheta : \nabla \Psi  \, \dx - \int_{\Omega_R} \zeta \cdot \nabla \vartheta \cdot  \Psi \, \dx  - \int_{\Omega_R} \vpres \nabla \cdot \Psi  \, \dx  \\  
& \qquad  \quad +  \int_{\Omega_R} (\vartheta \cdot \nabla) \tilde{h} \cdot \Psi \, \dx +  \int_\Omega (\tilde{h} \cdot \nabla) \vartheta \cdot \Psi \, \dx +  \int_{\Omega_R} (\vartheta \cdot \nabla) \vartheta \cdot \Psi \, \dx   \\ 
&  \qquad  \quad + \int_{\partial B_R} \frac 1R \left(1 + \wakefct(x) \right) \vartheta \cdot \Psi  \, \dS(x) - \int_{\partial B_R}  \frac 12 \left(\vartheta \cdot \frac{x}{R}\right) \tilde{h}  \cdot \Psi  \, \dS(x)  \\
& \qquad  \quad - \int_{\partial B_R}  \frac 12 \left(\tilde{h} \cdot \frac{x}{R}\right) \vartheta  \cdot \Psi  \, \dS(x)  - \int_{\partial B_R}  \frac 12 \left(\vartheta \cdot \frac{x}{R}\right) \vartheta  \cdot \Psi  \, \dS(x) \\
& \qquad  = \int_{\Omega_R} f \cdot  \Psi \, \dx - \int_{\Omega_R}  \partial_t \tilde{h} \cdot  \Psi \, \dx -  \int_\Omega \nabla \tilde{h} : \nabla \Psi  \, \dx    \\
& \qquad  \quad  +  \int_{\Omega_R} \zeta \cdot \nabla  \tilde{h}  \cdot  \Psi  \, \dx - \int_{\partial B_R} \frac 1R \left(1 + \wakefct(x) \right) \tilde{h}  \cdot \Psi  \, \dS(x) \\
& \qquad  \quad - \int_{\Omega_R}  (\tilde{h} \cdot \nabla) \tilde{h} \cdot \Psi \, \dx  + \int_{\partial B_R}  \frac 12 \left(\tilde{h} \cdot \frac{x}{R}\right) \tilde{h}  \cdot \Psi  \, \dS(x),  \quad  \forall  \Psi   \in W_R, \\
& \int_{\Omega_R} (\nabla \cdot \vartheta ) \phi \,  \dx   = 0,  \quad  \forall \phi \in L^2(\Omega_R),
\end{aligned}
\right.
\]
in the sense of distributions in ${\mathbb T}$.

It is convenient to recall \eqref{defc} and  introduce additional notations
\begin{equation}
\label{defab}
\begin{aligned}
   a(v,w)& :=  \int_{\Omega_R}  \nabla v : \nabla w \, \dx   -  \int_{\Omega_R}  \zeta \cdot \nabla v \cdot w \, \dx + \int_{\partial B_R}  \frac 1R \left(1 + \wakefct(x) \right) v \cdot w \, \dS(x),   \\
   b(v,p) & :=   - \int_{\Omega_R} ( \nabla \cdot v ) p \, \dx,   
\end{aligned}
\end{equation}
so that the above system for $(\vartheta,\vpres)$ can be reformulated in a more concise manner as 
\begin{equation}
\left\{
\begin{aligned}
 \left\langle \partial_t \vartheta, \Psi \right\rangle 
 & + a(\vartheta,\Psi) + b(\Psi,\vpres) + c(\vartheta,\tilde{h},\Psi)  + c(\tilde{h},\vartheta,\Psi)  + c(\vartheta,\vartheta,\Psi) 
 \\
& \quad 
= ( f, \Psi )_{\Omega_R} - \big( \partial_t \tilde{h}, \Psi  \big)_{\Omega_R} - a(\tilde{h},\Psi) -  c(\tilde{h},\tilde{h},\Psi) ,  \quad  \forall \Psi \in W_R, \\
   b(\vartheta,\phi) &= 0,  \quad  \forall \phi \in L^2(\Omega_R)
\end{aligned}
\right.
\label{wsolR}
\end{equation}
in $\torus$.
Moreover, we introduce a different inner product on the space $H^1(\Omega_R)^3$, namely,
\begin{equation}
\label{eq:innerproduct.zeta}
\begin{aligned}
(v,w)_{(R,|\zeta|)} := & \int_{\Omega_R}  \nabla v : \nabla w \, \dx +  \left(  \frac 1R + \frac{|\zeta|}{2}\right) \int_{\partial B_R}  (v \cdot w) \, \dS \\
= & \, 
a(v,w) +  \int_{\Omega_R}  \zeta \cdot \nabla v \cdot w \, \dx - \int_{\partial B_R}  \frac 1R  \frac{(\zeta \cdot x)}{2} (v \cdot w) \, \dS, 
\end{aligned}
\end{equation}
so that $
a(v,v) =  \|  \nabla v \|^2_{2,\Omega_R} + \int_{\partial B_R}  \left(  \frac 1R + \frac{|\zeta|}{2}\right) |v |^2 \, \dS(x) = \| v \|^2_{(R,|\zeta|)}$, for $v \in V_R.$
\begin{thm}
Let $f \in L^2({\mathbb T}\times \Omega_R)^3$ and 
$h \in H^1({\mathbb T};H^{\sfrac{1}{2}}(\partial \Omega)^3)$  
satisfying
\begin{equation}
\label{hipw}
2 \|  \Phi  \|_{\infty,{\mathbb T}} \left( C_S \| \sigma \|_{3,\Omega_R}  +  \frac{1}{8 \pi R } \right)  < 1.
\end{equation}
Then there exist  $v \in L^2({\mathbb T};H^1(\Omega_R)^3)\cap L^\infty({\mathbb T};L^2(\Omega_R)^3)$, $ \vpres_0  \in  L^{\infty}({\mathbb T},L^2(\Omega_R))$, $ \vpres_1 \in L^2({\mathbb T};L^2(\Omega_R))$ and $ \vpres_2 \in L^{4/3}({\mathbb T};L^2(\Omega_R))$ such that, in the sense of distributions in ${\mathbb T}$, it holds
\begin{equation}
\label{eq:weak.form}
\left\{
\begin{aligned}
& \frac{\mathrm{d}}{\dt} \left(  \int_{\Omega_R} v \cdot  \Psi  \, \dx  +  \int_{\Omega_R}  \vpres_0  \nabla \cdot  \Psi  \, \dx  \right) \\
& \,\,+  \int_{\Omega_R} \nabla v : \nabla \Psi  \, \dx  - \int_{\Omega_R} \zeta \cdot \nabla v \cdot  \Psi  \dx + \int_{\Omega_R} (v \cdot \nabla) v \cdot \Psi  \, \dx \\  
& \,\, + \int_{\partial B_R} \frac 1R \left(1 + \wakefct(x) \right) v \cdot \Psi  \, \dS(x)  - \int_{\partial B_R}  \frac 12 \left(v \cdot \frac{x}{R}\right) v  \cdot \Psi \, \dS(x) \\
& \, \, +  \int_{\Omega_R}  (\vpres_1 + \vpres_2)  \nabla \cdot  \Psi  \, \dx  
=   \int_{\Omega_R} f \cdot \Psi \dx ,  \quad  \forall \Psi  \in W_R, 
   \\
 & \int_{\Omega_R} (\nabla \cdot v) \phi \, \dx   = 0,  \quad  \forall \phi \in L^2(\Omega_R),
 \end{aligned}
\right.
\end{equation}
and $\vartheta := v - \tilde{h}$,
where $\tilde h$ from Lemma~\ref{div}, satisfies the energy inequality 
\begin{equation}
\begin{aligned}
& \int_\torus\int_{\Omega_R}|\nabla \vartheta|^2\,\dx\dt
+ \left( \frac{1}{R} + \frac{|\zeta|}{2} \right) \int_{\partial B_R}|\vartheta|^2\,\dS\dt \\
\leq &
 - \int_{\torus} \int_{\Omega_R}  \nabla \tilde{h} : \nabla \vartheta \, \dx  \dt  +  \int_{\torus}  \int_{\Omega_R}  \zeta \cdot \nabla \tilde{h} \cdot \vartheta \, \dx \dt \\ & -  \int_{\torus \times \partial B_R}  \frac 1R \left(1 + \wakefct(x) \right) (\tilde{h} \cdot \vartheta) \, \dS(x) \dt \\
&+ \int_{\torus} \int_{\Omega_R} \left( f  - \partial_t \tilde{h} - v \cdot \nabla \tilde{h} \right) \cdot \vartheta \, \dx \dt +  \frac 12  \int_{\torus \times \partial B_R}  \left(\frac xR \cdot v \right) (\tilde{h} \cdot \vartheta) \, \dS(x) \dt 
\end{aligned}
\label{enin}
\end{equation}
Moreover, if another weak solution $(\tilde{\vvel}, \tilde{\vpres})$ 
with $\tilde{v} \in  H^1(\mathbb T\times\Omega_R)^3$ exists such that 
\begin{equation}
\label{eq:smallness.uniqueness}
\|  \tilde{v} \|_{L^\infty({\mathbb T};H^1(\Omega_R))}
\leq \delta
\end{equation} 
with $\delta >0$ sufficiently small, then $\tilde{v} \equiv v$.
\label{thm:weakR} 
 \end{thm}

 \begin{proof} We construct a time-periodic weak solution to problem \eqref{wsolR} using the Galerkin method. In order to find the velocity $\vartheta$, let $\{\Psi_i\}_{i \in \mathbb N} \subset V_R$ be the complete orthonormal system in $H_R$ given by Lemma \ref{basis}. For each $M \in {\mathbb N}$, let $H^{(M)}_{R}$ be the linear space generated by $\{\Psi_1,...,\Psi_{M} \}$ endowed with the inner product of $H_R$, and let $V^{(M)}_{R}$ be defined in an analogous way with respect to the inner product of $V_R$.

In a first stage, approximate velocities $\vartheta^{(M)} \in L^\infty({\mathbb T};H^{(M)}_{R}) \cap L^2({\mathbb T};V^{(M)}_{R}) $ will be sought in the form 
\begin{equation}
 \vartheta^{(M)} (t,x)  =  \sum_{i=1}^{M} \alpha_{i}(t) \Psi_i(x), \quad \alpha_{ j} \in W^{1,\sfrac{4}{3}}({\mathbb T}).
 \label{vM}
\end{equation}

In order to determine the $T$-periodic functions $\alpha_{1},...,\alpha_{M}$, let ${\mathcal F} :  {\mathbb T} \times  {\mathbb R}^{M}  \to  {\mathbb R}^{M}$ with components
\[
\begin{aligned}
{\mathcal F}_{m}(t,\alpha)  = & - \sum_{i=1}^{M} \alpha_{ i} \left[ a(\Psi_i,\Psi_m)  + c(\Psi_i,\tilde{h},\Psi_m)  + c(\tilde{h},\Psi_i,\Psi_m)  \right]    \\
  & - \sum_{i,j=1}^{M} \alpha_{ i} \alpha_{ j} c(\Psi_i,\Psi_j,\Psi_m)   \\
& + \big( f - \partial_t \tilde{h}, \Psi_m \big)_{\Omega_R} - a(\tilde{h},\Psi_m) -  c(\tilde{h},\tilde{h},\Psi_m) ,  \quad  m = 1,...,M, 
\end{aligned}
\]
where 
\[
\big( f - \partial_t \tilde{h}, \Psi_m \big)_{\Omega_R}  \in L^2(\mathbb T), \quad \,  a(\tilde{h},\Psi_m) , \, c(\tilde{h},\tilde{h},\Psi_m)  \in  C(\mathbb T), \quad m = 1,...,M.
\]
 Then \eqref{vM}, more specifically  $\alpha = (\alpha_{1},...,\alpha_{M})$, will be obtained as a $T$-periodic solution of the systems of ODEs
 \begin{equation}
\alpha' = {\mathcal F}  (t,\alpha) \text{ in } {\mathbb T}.
\label{coefalfa}
\end{equation}

At this stage, $M \in {\mathbb N}$ is fixed. For a fixed $\underline{\alpha} \in W^{1,\sfrac{4}{3}}({\mathbb T})^M$, consider the linearized problem
\begin{equation}
\alpha' = {\mathcal L}  (t,\alpha;\underline{\alpha}) \text{ in } {\mathbb T}.
\label{coefalfaL}
\end{equation}
where 
\[
\begin{aligned}
{\mathcal L}(\,\cdot,{\cdot}\,;\underline{\alpha}) & :  {\mathbb T} \times  {\mathbb R}^{M}  \to  {\mathbb R}^{M}, \\
{\mathcal L}_{m}(t,\alpha;\underline{\alpha})  & =  - \sum_{i=1}^{M} \alpha_{ i} A_{im} - \sum_{i=1}^{M} \underline{\alpha}_{ i} \left[  c(\Psi_i,\tilde{h},\Psi_m)  + c(\tilde{h},\Psi_i,\Psi_m)  \right]    \\
  &\quad - \sum_{i,j=1}^{M}\underline{ \alpha}_{ i} \underline{\alpha}_{ j} c(\Psi_i,\Psi_j,\Psi_m) + g_m(t),  \qquad  m = 1,...,M, 
\end{aligned}
\]
and
\[
\begin{aligned}
  A_{im} &:=  \, a(\Psi_i,\Psi_m), 
  &  i,m &= 1,...,M,  \\
g_m(t) &:= \big( f - \partial_t \tilde{h}, \Psi_m \big)_{\Omega_R} - a(\tilde{h},\Psi_m) -  c(\tilde{h},\tilde{h},\Psi_m),  
& m &= 1,...,M.
\end{aligned}
\]

In order to  alleviate the presentation, we put 
\[
\psi_0(t):=1, \quad \psi_k^c(t) := \sqrt{2}  \cos \left( \frac{2\pi}{\per} k t \right),\quad \psi_k^s(t) := \sqrt{2}  \sin \left( \frac{2\pi}{\per} k t \right), \quad k \in {\mathbb N},
\]
and recall the orthonormality relations for $\left\{ \psi_0 , \, \psi_k^c, \, \psi_k^s :  k \in \mathbb N  \right\}$ in $L^2(\mathbb T)$.
 A solution for the  system of ODEs \eqref{coefalfaL} can be sought in the form of a Fourier series
\begin{equation}
\alpha_{i}(t) = \alpha_{i 0}\psi_0(t) +  \sum_{k = 1}^\infty \alpha_{i k}^c\psi_k^c(t)  +   \sum_{k =1}^\infty \alpha_{i k}^s \psi_k^s(t)  ,  \quad  i =1,...,M.
\label{defindalfa}
\end{equation}
It is convenient to write  \eqref{coefalfaL} as 
\begin{equation}
\alpha ' + A \alpha =  G(\underline{\alpha}) \text{ in } {\mathbb T}
\label{alphaAg}
\end{equation}
with $A:=(A_{im})_{1 \leq i,m \leq M}$ and $G(\underline{\alpha})=(G_{m}(\underline{\alpha}))_{1 \leq m \leq M}$, where
\[
G_m(\underline{\alpha}) := \, g_m - \sum_{i=1}^{M} \underline{\alpha}_{ i} \left[  c(\Psi_i,\tilde{h},\Psi_m)  + c(\tilde{h},\Psi_i,\Psi_m)  \right]   - \sum_{i,j=1}^{M}\underline{ \alpha}_{ i} \underline{\alpha}_{ j} c(\Psi_i,\Psi_j,\Psi_m) \in L^{\sfrac{4}{3}}(\mathbb T).
\]
Based on \eqref{defindalfa}, we define
\[
\alpha_{0} := \begin{bmatrix}
    \alpha_{10}    \\
    \alpha_{20} \\
    \vdots \\
     \alpha_{M0}
\end{bmatrix}, \quad
\alpha_{k}^s := \begin{bmatrix}
    \alpha_{1k}^s    \\
    \alpha_{2k}^s  \\
    \vdots \\
     \alpha_{Mk}^s
\end{bmatrix}, \quad \alpha_{k}^c : = \begin{bmatrix}
    \alpha_{1k}^c    \\
    \alpha_{2k}^c  \\
    \vdots \\
     \alpha_{Mk}^c
\end{bmatrix}, \quad k \in {\mathbb N}
\]
and 
\[
G_{m0} := \int_{\mathbb T} G_m(t) \psi_0^s(t) \dt ,\quad G_{mk}^s := \int_{\mathbb T} G_m(t) \psi_k^s(t) \dt ,\quad G_{mk}^c:= \int_{\mathbb T} G_m(t) \psi_k^c(t) \dt,
\]
\[
G_{0} := \begin{bmatrix}
    G_{10}    \\
    G_{20} \\
    \vdots \\
    G_{M0}
\end{bmatrix}, \quad
G_{k}^s := \begin{bmatrix}
    G_{1k}^s    \\
    G_{2k}^s  \\
    \vdots \\
     G_{Mk}^s
\end{bmatrix}, \quad G_{k}^c : = \begin{bmatrix}
    G_{1k}^c    \\
    G_{2k}^c  \\
    \vdots \\
     G_{Mk}^c
\end{bmatrix}, \quad k \in {\mathbb N}.
\]
The Fourier coefficients of a solution $\alpha$ to \eqref{alphaAg} can be obtained by solving the sequence of linear systems
\begin{equation}
   A  \alpha_{0}   =  
    G_{0} ,
\qquad \begin{bmatrix}
  A  & - \frac{2\pi}{\per} k \, {\mathbb I}_M   \\
\frac{2\pi}{\per} k \, {\mathbb I}_M & A
\end{bmatrix}  \begin{bmatrix}
    \alpha_{k}^s    \\
    \alpha_{k}^c  
\end{bmatrix} = \begin{bmatrix}
    G_{k}^s    \\
    G_{k}^c  
\end{bmatrix} , \quad k  \in {\mathbb N}.
\label{bloc}
\end{equation}
Here, ${\mathbb I}_M$ is the identity matrix in ${\mathbb R}^{M \times M}$.
Note that the matrix $A \in {\mathbb R}^{M \times M}$ is positive definite since we have
\[
\sum_{i,m=1}^M z_i A_{im} z_m = a(\Psi,\Psi) = \| \nabla \Psi \|^2_{2,\Omega_R} + \int_{\partial B_R}  \left(  \frac{1}{R}  + \frac{|\zeta|}{2}\right) |\Psi|^2 \dS \quad (z \in {\mathbb R}^M, \, \Psi = z_i \Psi_i),
\]
and the block matrices in \eqref{bloc}, defined in terms of $A$ and $\frac{2\pi}{\per} k \, {\mathbb I}_M$, are nonsingular. For each $k  \in {\mathbb N}$,
\[
\begin{aligned}
 \begin{bmatrix}
    \alpha_{k}^s    \\
    \alpha_{k}^c  
\end{bmatrix} = & \begin{bmatrix}
  A  & - \frac{2\pi k}{\per} \, {\mathbb I}_M   \\
\frac{2\pi k}{\per}  \, {\mathbb I}_M & A
\end{bmatrix}^{-1}\begin{bmatrix}
    G_{k}^s    \\
    G_{k}^c  
\end{bmatrix} 
\\
= & \begin{bmatrix}
  \left(A^2 + \frac{4 \pi^2 k^2}{\per^2} {\mathbb I}_M\right)^{-1}A  & \frac{2\pi k}{\per}  \left(A^2 + \frac{4 \pi^2 k^2}{\per^2} {\mathbb I}_M\right)^{-1}    \\
- \frac{2\pi k}{\per}  \left(A^2 + \frac{4 \pi^2 k^2}{\per^2} {\mathbb I}_M\right)^{-1} & \left(A^2 + \frac{4 \pi^2 k^2}{\per^2} {\mathbb I}_M\right)^{-1} A
\end{bmatrix} \begin{bmatrix}
    G_{k}^s    \\
    G_{k}^c  
\end{bmatrix} 
\end{aligned}
\]
and, with $A_k :=  \frac{\per}{2\pi k} A$, for $k \in {\mathbb N}$, we have
\[
  \begin{bmatrix}
    \alpha_{k}^s    \\
    \alpha_{k}^c  
\end{bmatrix}   = \frac{\per}{2 \pi k} \begin{bmatrix}
  (A_k^2 + {\mathbb I}_M)^{-1}A_k    G_{k}^s  +  (A_k^2 +  {\mathbb I}_M)^{-1} G_{k}^c    \\
- (A_k^2 + {\mathbb I}_M)^{-1} G_{k}^s  +  (A_k^2 +  {\mathbb I}_M)^{-1} A_k G_{k}^c 
\end{bmatrix} , \, k \in {\mathbb N}.
\]
For any matrix norm $\norm{\cdot}$, there exists a constant $C>0$ such that $\| (A_k^2 + {\mathbb I}_M)^{-1} \| \leq C$, for all $k > \|A\|$.
By Hausdorff-Young inequality, we have $\{G_{k}^s\}_{k \in {\mathbb N}} ,\{G_{k}^c\}_{k \in {\mathbb N}} \in \ell^{4}({\mathbb N})^M$ and therefore, $\{k \alpha_{k}^s \}_{k \in {\mathbb N}}, \{k \alpha_{k}^c \}_{k \in {\mathbb N}} \in \ell^4({\mathbb N})^M$. 
By H\"{o}lder inequality, we obtain $\{\alpha_{k}^s \}_{k \in {\mathbb N}}, \{\alpha_{k}^c \}_{k \in {\mathbb N}} \in \ell^{\sfrac{4r}{(4+r)}}({\mathbb N})^M$, for all $r>\sfrac{4}{3}$. Thus $\{\alpha_{k}^s \}_{k \in {\mathbb N}}, \{\alpha_{k}^c \}_{k \in {\mathbb N}} \in \ell^{r}({\mathbb N})^M$, for all $1 < r < 4$. This, in turn, yields the existence of $\alpha \in L^2(\mathbb T)^M$ solving~\eqref{coefalfaL}, and from the identity \eqref{alphaAg}, it follows that $\alpha' \in L^{\sfrac{4}{3}}(\mathbb T)^M$.

We can thus consider the mapping 
\[ 
\begin{aligned}
& {\mathcal M} : W^{1,\sfrac{4}{3}}({\mathbb T})^M \to W^{1,\sfrac{4}{3}}({\mathbb T})^M \\
& {\mathcal M}(\underline{\alpha}) = \alpha.
\end{aligned}
\] 
Our aim is to establish existence of a fixed point of $\mathcal M$. 

In order to use the Leray--Schauder fixed-point Theorem, we first show that the solution of the problem
\begin{equation}
\alpha ' + A \alpha =  \lambda G(\alpha) \text{ in } {\mathbb T}
\label{alphaAglam}
\end{equation}
are uniformly bounded with respect to $\lambda \in [0,1]$. By taking the dot product of both sides of equation \eqref{alphaAglam} with $\alpha$, we obtain
\[
\begin{aligned}
\frac{1}{2} \frac{d}{dt} |\alpha|^2 = & - \sum_{i,m=1}^{M} \alpha_{ i} A_{im}  \alpha_m - \lambda \sum_{i,m=1}^{M} \alpha_{ i} \alpha_m \left[  c(\Psi_i,\tilde{h},\Psi_m)  + c(\tilde{h},\Psi_i,\Psi_m)  \right]    \\
  & - \lambda \sum_{i,j,m=1}^{M} \alpha_{ i}  \alpha_{ j} \alpha_m c(\Psi_i,\Psi_j,\Psi_m) + \lambda \sum_{m=1}^{M}  g_m \alpha_m , \quad \lambda \in [0,1].
\end{aligned}
\]
Recalling \eqref{vM} and using the orthonormality conditions that $\left\{ \psi_0 , \, \psi_k^c, \, \psi_k^s :  k \in \mathbb N  \right\}$ and $\{\Psi_1,...,\Psi_{M} \}$ induce in $L^2(\mathbb T,H^{(M)}_{R})$, we get
\begin{equation}
\begin{aligned}
\frac{1}{2} \frac{d}{dt} \| \vartheta^{(M)} \|_{2,\Omega_R}^2  & +   \, a(\vartheta^{(M)},\vartheta^{(M)})  + \lambda  c(\vartheta^{(M)},\tilde{h},\vartheta^{(M)})    \\
   &  +  \,  \lambda c(\tilde{h},\vartheta^{(M)},\vartheta^{(M)}) + \lambda c(\vartheta^{(M)},\vartheta^{(M)},\vartheta^{(M)})   \\
= &\, \, \lambda \big( f  - \partial_t \tilde{h}, \vartheta^{(M)} \big)_{\Omega_R} - \lambda  a(\tilde{h},\vartheta^{(M)}) - \lambda  c(\tilde{h},\tilde{h},\vartheta^{(M)}) .
\end{aligned}
\label{eq:ddttheta.norm}
\end{equation}
Since, by the time-periodicity of $\vartheta^{(M)}$ and by \eqref{cigual0},
it holds
\[
\int_{\mathbb T}  \frac{d}{dt} \| \vartheta^{(M)} \|_{2,\Omega_R}^2 \dt  =  0, \quad c(\tilde{h},\vartheta^{(M)},\vartheta^{(M)}) =   c(\vartheta^{(M)},\vartheta^{(M)},\vartheta^{(M)}) =0,
\] 
and, by direct calculation, 
\[
c(\vartheta^{(M)},\tilde{h},\vartheta^{(M)})  = -  c(\vartheta^{(M)},\vartheta^{(M)},\tilde{h}),
\] 
we obtain
\begin{equation}
\begin{aligned}
\| \vartheta^{(M)}  \|_{({\mathbb T},R)}^2 \leq & \, \int_{\mathbb T} \left( \| \vartheta^{(M)} \|_{(R)}^2  + \frac{|\zeta|}{2}\| \vartheta^{(M)} \|^2_{2,\partial B_R} \right) \dt    \\
= & \,   \int_{\mathbb T}  a( \vartheta^{(M)},  \vartheta^{(M)}) \dt
   \\
 = & \,  \lambda  \int_{\mathbb T}  c(\vartheta^{(M)}, \vartheta^{(M)},\tilde{h}) \dt + \lambda  \int_{\mathbb T} \left( f - \partial_t \tilde{h}, \vartheta^{(M)}  \right)_{\Omega_R} \dt   \\
 & \, - \lambda  \int_{\mathbb T} a(\tilde{h},\vartheta^{(M)} ) \dt - \lambda   \int_{\mathbb T}  c(\tilde{h},\tilde{h},\vartheta^{(M)} ) \dt.
 \end{aligned}
\label{ineqvll00}
\end{equation}
From Lemma  \ref{tdiv}, estimate \eqref{estdiv}, we conclude
\[
\begin{aligned} 
 \lambda  \int_{\mathbb T}  c(\vartheta^{(M)}, \vartheta^{(M)},\tilde{h}) \dt  \leq  \gamma  \| \vartheta^{(M)}  \|_{({\mathbb T},R)}^2  + \|  \Phi  \|_{\infty,{\mathbb T}} \left[ C_S \| \sigma \|_{3,\Omega_R}  +  \frac{1}{8 \pi  R } \right]  \|  \vartheta^{(M)}  \|_{({\mathbb T},R)}^2,
 \end{aligned}
\]
and since $1 - \|  \Phi  \|_{\infty,{\mathbb T}} \left[ C_S \| \sigma \|_{3,\Omega_R}  +  \frac{1}{8 \pi  R } \right]  > \gamma >0$, by estimating the remaining terms in the last equality of \eqref{ineqvll00}, after estimating the remaining terms on the last equality, we arrive at 
\begin{equation}
\| \vartheta^{(M)} \|_{({\mathbb T},R)}  \leq \frac{C (\Omega_R) \left[ \| f \|_{L^2({\mathbb T}\times \Omega)} + (1 + |\zeta|)\| \tilde{h} \|_{H^1({\mathbb T} \times \Omega)}  + \| \tilde{h} \|_{H^1({\mathbb T} \times \Omega)}^2  \right] }{\left(1-\|  \Phi  \|_{\infty,{\mathbb T}} \left[C_S \| \sigma \|_{3,\Omega_R}  +  \frac{1}{8 \pi  R } \right] -\gamma \right)^{\sfrac{1}{2}}}.
\label{ineqvll0}
\end{equation} 
By Poincar\'e inequality and the orthonormality conditions in $L^2(\mathbb T,H^{(M)}_{R})$, this implies
\[
\begin{aligned}
 \|\alpha\|_{L^2({\mathbb T})^M} = & \, \| \vartheta^{(M)} \|_{L^2({\mathbb T}; H_R)}   \\
\leq & \, C_P(\Omega_R) \| \vartheta^{(M)} \|_{L^2({\mathbb T}; V_R)} = \, C_P(\Omega_R)  \| \vartheta^{(M)} \|_{({\mathbb T},R)}     \\
\leq & \, \frac{C_P(\Omega_R)  C(\Omega_R) \left[ \| f \|_{L^2({\mathbb T}\times \Omega)} + (1+|\zeta|)\| \tilde{h} \|_{H^1({\mathbb T} \times \Omega)}  + \| \tilde{h} \|_{H^1({\mathbb T} \times \Omega)} ^2  \right] }{\left( 1-\|  \Phi  \|_{\infty,{\mathbb T}} \left[ C_S \| \sigma \|_{3,\Omega_R}  +  \frac{1}{8 \pi  R } \right] - \gamma \right)^{\sfrac{1}{2}}},
\end{aligned}
\]
where $C_P (\Omega_R)$ is a Poincar\'e constant on $\Omega_R$. Then, going back to \eqref{alphaAglam}, we conclude that $\alpha \in W^{1,\sfrac{4}{3}}({\mathbb T})^M$ and $\alpha'$ is also bounded by the data in $L^{\sfrac{4}{3}}({\mathbb T})^M$.

Now, we show that the mapping $\mathcal M$ is compact. Suppose that the sequence $\{ \underline{\alpha}^{(k)} \}_{k \in \mathbb N} \subset W^{1,\sfrac{4}{3}}({\mathbb T})^M$ is bounded. We have
\begin{equation}
(\alpha^{(k)}  - \alpha^{(\ell)})' + A (\alpha^{(k)} - \alpha^{(\ell)} ) =  G(\underline{\alpha}^{(k)}) - G(\underline{\alpha}^{(\ell)}) \text{ in } {\mathbb T}
\label{alphaAgcomp}
\end{equation}
where, for each $m \in \{1,...,M\}$, 
\begin{equation}
\begin{aligned}
& G_m(\underline{\alpha}^{(k)}) - G_m(\underline{\alpha}^{(\ell)})    \\
= & \,  \sum_{i=1}^{M} ( \underline{\alpha}_i^{(\ell)} - \underline{\alpha}_{ i}^{(k)} ) \left[  c(\Psi_i,\tilde{h},\Psi_m)  + c(\tilde{h},\Psi_i,\Psi_m)  \right]    \\
& + \sum_{i,j=1}^{M}(\underline{ \alpha}_{ i}^{(\ell)} \underline{\alpha}_{ j}^{(\ell)} - \underline{ \alpha}_{ i}^{(k)} \underline{\alpha}_{ j}^{(k)} ) c(\Psi_i,\Psi_j,\Psi_m)   \\
=  & \,  \sum_{i=1}^{M} ( \underline{\alpha}_i^{(\ell)} - \underline{\alpha}_{ i}^{(k)} ) \left[  c(\Psi_i,\tilde{h},\Psi_m)  + c(\tilde{h},\Psi_i,\Psi_m)  \right]    \\
& + \sum_{i,j=1}^{M}\left[ \underline{\alpha}_{ j}^{(\ell)} (\underline{ \alpha}_{ i}^{(\ell)} - \underline{ \alpha}_{ i}^{(k)} )  + \underline{ \alpha}_{ i}^{(k)} (\underline{\alpha}_{ j}^{(\ell)} - \underline{\alpha}_{ j}^{(k)}) \right] c(\Psi_i,\Psi_j,\Psi_m).
\end{aligned}
\label{alfaG}
\end{equation}
The embedding $W^{1,\sfrac{4}{3}}({\mathbb T})^M  \hookrightarrow  C({\mathbb T})^M$ is compact, hence $\{ \underline{\alpha}^{(k)} \}_{k \in \mathbb N}$ contains a subsequence $\{ \underline{\alpha}^{(k')} \}_{k' \in \mathbb N}$ that converges in $C({\mathbb T})^M$. Let
\begin{equation}
 \underline{\vartheta}^{(M,k')} (t,x) : =  \sum_{i=1}^{M} \underline{\alpha}^{(k')}_{i}(t) \Psi_i(x), \quad \vartheta^{(M,k')} (t,x) : =  \sum_{i=1}^{M} \alpha^{(k')}_{i}(t) \Psi_i(x).
 \label{vMk}
\end{equation}
The sequence $\{ \underline{\vartheta}^{(M,k')} \}_{k' \in \mathbb N}$ converges in $L^2({\mathbb T}; V_R^{(M)})$ and therefore it is a Cauchy sequence in $L^2({\mathbb T}; V_R)$. Taking the dot product of both sides of \eqref{alphaAgcomp}  with $\alpha^{(k)} - \alpha^{(\ell)} $ and recalling \eqref{alfaG},  we get 
\[
\begin{aligned}
 \|  \vartheta^{(M,k')} -  \vartheta^{(M,\ell')} & \|^2_{L^2({\mathbb T}; V_R)}  
\leq \int_{\mathbb T}  a\big(\vartheta^{(M,k')} -  \vartheta^{(M,\ell')},\vartheta^{(M,k')} -  \vartheta^{(M,\ell')}\big) \dt
  \\
\qquad\leq  C (\Omega_R,M) &\Big( \| \tilde{h} \|_{H^1({\mathbb T} \times \Omega)}   \|  \vartheta^{(M,k')} -  \vartheta^{(M,\ell')}  \|_{L^2({\mathbb T}; V_R)} 
  \|  \underline{\vartheta}^{(M,k')} -  \underline{\vartheta}^{(M,\ell')}  \|_{L^2({\mathbb T}; V_R)}  
  \\ 
 &\quad +  \big(\|  \underline\vartheta^{(M,k')}\|_{L^\infty(\torus;V_R)} + \|  \underline\vartheta^{(M,\ell')}\|_{L^\infty(\torus;V_R)}   \big)
  \\
  &\qquad\quad\times 
   \|  \vartheta^{(M,k')} - \vartheta^{(M,\ell')}  \|_{L^2({\mathbb T}; V_R)} 
  \|  \underline{\vartheta}^{(M,k')} -  \underline{\vartheta}^{(M,\ell')}  \|_{L^2({\mathbb T}; V_R)}  \Big).
\end{aligned}
\] 
As in previous estimates, by Poincar\'e inequality and the orthonormality conditions in $L^2(\mathbb T,H^{(M)}_{R})$, we get
\[
\begin{aligned}
 \|\alpha^{(k')}  - \alpha^{(\ell')}\|_{L^2({\mathbb T})^M} = & \, \|  \vartheta^{(M,k')} -  \vartheta^{(M,\ell')} \|_{L^2({\mathbb T}; H_R)}   \\
\leq & \, C_P(\Omega_R) \|  \vartheta^{(M,k')} -  \vartheta^{(M,\ell')}  \|_{L^2({\mathbb T}; V_R)},
\end{aligned}
\]
and now we use the fact that $\{\underline\vartheta^{(M,k')} \}_{k' \in \mathbb N}$ is a Cauchy sequence in $L^2({\mathbb T}; V_R)$ to conclude that $\{ \alpha^{(k')} \}_{k' \in \mathbb N}$ is a Cauchy sequence in $L^2({\mathbb T})^M$. From \eqref{alphaAgcomp} and the previous estimates, we also get 
\[
\|(\alpha^{(k')})'  - (\alpha^{(\ell')})'\|_{L^{\sfrac{4}{3}}({\mathbb T})^M} \leq \|A (\alpha^{(k')} - \alpha^{(\ell')} )\|_{L^{\sfrac{4}{3}}({\mathbb T})^M} + \|G(\underline{\alpha}^{(k')}) - G(\underline{\alpha}^{(\ell')}) \|_{L^{\sfrac{4}{3}}({\mathbb T})^M}.
\]
Using the strong convergence of 
$\{ \underline{\alpha}^{(k')} \}_{k' \in \mathbb N}$ 
in $C({\mathbb T})^M$,
we conclude that 
$\mathcal M$ maps bounded sequences into relatively compact ones. 
In conclusion, the Leray--Schauder Theorem shows that the mapping $\mathcal M$ has a fixed point.

We thus solved  \eqref{coefalfa} with $M$ fixed and obtained an approximate solution \eqref{vM} which satisfies \eqref{ineqvll0}. Now, we derive additional estimates for the sequence $\{ \vartheta^{(M)} \}_{M \in {\mathbb N}}$. Actually, $ \vartheta^{(M)} \in C({\mathbb T};V_R)$, and from  the estimate \eqref{ineqvll0} and the mean value theorem for continuous functions, 
we conclude the existence of $\overline{t} \in (0,\mathcal T)$ such that
\begin{equation}
\begin{aligned}
\| \vartheta^{(M)} (\overline{t}) \|^2_{(R)} & = \int_{{\mathbb T}} \| \vartheta^{(M)} \|^2_{(R)} \dt = \| \vartheta^{(M)} \|^2_{({\mathbb T},R)}    \\
& \leq \frac{C(\Omega_R)^2 \left[ \| f \|_{L^2({\mathbb T}\times \Omega)} + (1+|\zeta|)\| \tilde{h} \|_{H^1({\mathbb T} \times \Omega)}  + \| \tilde{h} \|_{H^1({\mathbb T} \times \Omega)}^2  \right]^2}{1-\|  \Phi  \|_{\infty,{\mathbb T}} \left[C_S \| \sigma \|_{3,\Omega_R}  +  \frac{1}{8 \pi  R } \right] -\gamma }.
\end{aligned}
\label{eq:thetaM.tbarest}
\end{equation}
From Poincar\'e inequality, we further get
\begin{equation}
\| \vartheta^{(M)} (\overline{t}) \|_{2,\Omega_R}  \leq C_P (\Omega_R) \| \vartheta^{(M)} (\overline{t}) \|_{(R)}.
\label{initialt}
\end{equation}
Now, on the time interval $[\overline{t}, \overline{t} + \per]$, we consider \eqref{eq:ddttheta.norm} (with $\lambda=1$).
Taking into account~\eqref{cigual0}, we have
\begin{equation}
\begin{aligned}
& \frac 12 \frac{d}{dt} \| \vartheta^{(M)} \|_{2,\Omega_R}^2   +   a(\vartheta^{(M)},\vartheta^{(M)})     \\
&\quad = c(\vartheta^{(M)},\vartheta^{(M)},\tilde{h})  +  \big( f  - \partial_t \tilde{h}, \vartheta^{(M)} \big)_{\Omega_R} -  a(\tilde{h},\vartheta^{(M)}) -  c(\tilde{h},\tilde{h},\vartheta^{(M)}),
\end{aligned}
\label{ttoT}
\end{equation}
where
\[
a(\vartheta^{(M)},\vartheta^{(M)})  =  \| \nabla \vartheta^{(M)} \|^2_{2,\Omega_R} + \int_{\partial B_R} \left( \frac{1}{R}  + \frac{|\zeta|}{2}\right) |\vartheta^{(M)}|^2 \dS \geq \| \vartheta^{(M)} \|_{(R)}^2,
\]
and
\[
\begin{aligned}
\big| \big( f  - \partial_t \tilde{h}, \vartheta^{(M)} \big)_{\Omega_R} \big| & \leq \frac 14  \| \vartheta^{(M)} \|^2_{2,\Omega_R}  + 2 \|  f \|^2_{2,\Omega_R} + 2 \| \partial_t \tilde{h}\|^2_{2,\Omega_R},  \\
| a(\tilde{h},\vartheta^{(M)}) | & \leq  \, C(\Omega_R)(1 + |\zeta|) \| \tilde{h}\|_{1,2,\Omega_R} \| \vartheta^{(M)} \|_{(R)} + |\zeta| \| \tilde{h}\|_{1,2,\Omega_R} \| \vartheta^{(M)} \|_{2,\Omega_R} \\
& \leq  \,  \frac 14 \| \vartheta^{(M)} \|_{(R)}^2  + \frac 14 \| \vartheta^{(M)} \|^2_{2,\Omega_R} + C(\Omega_R)(1 + |\zeta|)^2 \| \tilde{h}\|^2_{1,2,\Omega_R} , \\
|c(\tilde{h},\tilde{h},\vartheta^{(M)})| &  \leq \frac 14 \| \vartheta^{(M)} \|^2_{(R)} + C(\Omega_R) \| \tilde{h} \|^4_{1,2,\Omega_R} .
\end{aligned}
\]
We estimate the term $ c(\vartheta^{(M)},\vartheta^{(M)},\tilde{h})$ using Lemma \ref{tdiv}, apply the preceding estimates,
and use assumption \eqref{hipw} to deduce
\[
\begin{aligned}
 \frac{d}{dt} \| \vartheta^{(M)} \|_{2,\Omega_R}^2   + &  \, \left[1 - 2 \|  \Phi  \|_{\infty,{\mathbb T}} \left( C_S \| \sigma \|_{3,\Omega_R}  +  \frac{1}{8 \pi  R } \right) - 2 \gamma \right] \| \vartheta^{(M)} \|^2_{(R)}    \\
\leq & \,  
 \| \vartheta^{(M)}  \|_{2,\Omega_R}^2  +  C_2(\Omega_R) (1 + |\zeta|)^2 \| \tilde{h}\|^2_{1,2,\Omega_R}  +  C_3(\Omega_R) \| \tilde{h} \|^4_{1,2,\Omega_R} \\
& \, + 4 \|  f \|^2_{2,\Omega_R} + 4 \| \partial_t \tilde{h}\|^2_{2,\Omega_R} ,
\end{aligned}
\]
which we combine with \eqref{initialt} and the estimate~\eqref{eq:thetaM.tbarest} for $\| \vartheta^{(M)} (\overline{t}) \|_{(R)}$. The Gr\"onwall Lemma and the time-periodicity of $\vartheta^{(M)}$ yield
\begin{equation}\label{eq:thetaM.Linfty}
\| \vartheta^{(M)}  \|_{L^\infty(\mathbb T;L^{2}(\Omega_R))} \leq  C\left(\Omega_R, \mathcal T, |\zeta|,\| f \|_{L^2(\mathbb T;L^{2}(\Omega_R))},\| \tilde{h} \|_{H^1({\mathbb T}; H^1(\Omega))}\right).
\end{equation}

An estimate for the time derivative of $\vartheta^{(M)}$ can be obtained as follows: for each $M \in {\mathbb N}$, let ${\mathbb P}_{M}$ be the orthogonal projector onto $\textrm{span}\{\Psi_1,...,\Psi_{M}\}$ in $V_R$. Recall that, by Lemma \ref{basis}, we have, for each $\Phi  \in V_R$,
\[
\|{\mathbb P}_{M} \Phi  \|_{(R)} \leq \| \Phi  \|_{(R)}, \qquad
{\mathbb P}_{M} \Phi  \to \Phi  \textrm{ in } V_R \text{ as } M \to \infty.
\]
Since $\{\Psi_k\}_{k \in \mathbb N}$ is a complete orthonormal system  in $H_R$, we have
$$
 \int_{\Omega_R}  \partial_t  \vartheta^{(M)} \cdot  \Phi   \, \dx
  =  \int_{\Omega_R}  \partial_t  \vartheta^{(M)} \cdot  ( {\mathbb P}_{M} \Phi ) \, \dx, \quad \forall \Phi \in H_R 
$$
and therefore
\[
\begin{aligned}
 & \big( \partial_t \vartheta^{(M)}, \Phi  \big)_{\Omega_R}   + a(\vartheta^{(M)} , {\mathbb P}_{M} \Phi )  + c(\vartheta^{(M)} ,\tilde{h}, {\mathbb P}_{M} \Phi )    \\
  &\quad + c(\tilde{h},\vartheta^{(M)} , {\mathbb P}_{M} \Phi )  + c(\vartheta^{(M)} , \vartheta^{(M)} , {\mathbb P}_{M} \Phi )   \\
  &\qquad =  \left( f, {\mathbb P}_{M} \Phi \right)_{\Omega_R}  - \big(\partial_t \tilde{h},  {\mathbb P}_{M} \Phi  \big)_{\Omega_R}  - a(\tilde{h}, {\mathbb P}_{M} \Phi ) -  c(\tilde{h},\tilde{h}, {\mathbb P}_{M} \Phi  ) ,  \quad \forall \Phi  \in V_R ,
 \end{aligned}
\]
which, by setting 
\[
\begin{aligned}
&\langle {\mathcal A}v,w \rangle_{V_R',V_R}  :=    a(v,w), \\
&\langle {\mathcal C}(u,v),w \rangle_{V_R',V_R}  := c(u,v,w),  
\end{aligned}
\]
can be written as 
\begin{multline*}
 \langle \partial_t \vartheta^{(M)} , \Psi  \rangle_{V_R',V_R} = - \langle {\mathcal A} v^{(M)} + {\mathcal C}(v^{(M)} ,v^{(M)} ), {\mathbb P}_{M} \Psi  \rangle_{V_R',V_R}   \\ + \langle f - \partial_t \tilde{h},{\mathbb P}_{M} \Psi  \rangle_{V_R',V_R} , \quad \forall \Psi  \in V_R, \mbox{ in } {\mathbb T},
\end{multline*}
where $v^{(M)}=\vartheta^{(M)}+\tilde h$.
By interpolation, from~\eqref{ineqvll0} and~\eqref{eq:thetaM.Linfty} we deduce that $\{v^{(M)}\}_{M\in\N}\subset L^4({\mathbb T};L^3(\Omega_R))$ 
is uniformly bounded. Since 
\begin{equation}
\begin{aligned}
\|{\mathcal A}(v)\|_{V_R'} 
&\leq C(\Omega_R)(1 + |\zeta|) \| v\|_{1,2,\Omega_R}, \qquad \forall v \in H^1(\Omega_R),  
\\
\|{\mathcal C}(v,v) \|_{V_R'} 
&\leq C(\Omega_R) 
\| v\|_{3;\Omega}\| v\|_{1,2;\Omega} ,
\qquad \forall v \in H^1(\Omega_R),
\\
\| \partial_t \vartheta^{(M)}  \|_{V_R'} 
&\leq \|{\mathcal A}(v^{(M)} )\|_{V_R'}  +  \|{\mathcal C}(v^{(M)},v^{(M)}  )\|_{V_R'}  +  \| f \|_{2,\Omega_R} + \| \partial_t \tilde{h} \|_{2,\Omega_R},
\end{aligned}
\label{normaAC}
\end{equation}
we conclude that $\{ \partial_t \vartheta^{(M)}  \}_{M \in {\mathbb N}}$ remains in a bounded set of $L^{4/3}({\mathbb T};V_R')$. 
Here, we used that, by integration by parts and Sobolev embeddings,
we have
\[
c(u,v,w)
=\frac{1}{2}\int_{\Omega_R} u\cdot \nabla v\cdot w\,\dx 
-\frac{1}{2}\int_{\Omega_R} u\cdot \nabla w\cdot v\,\dx 
\leq C(\Omega_R) \|u\|_{3;\Omega}\|v\|_{1,2;\Omega}\|w\|_{1,2;\Omega}
\]
if $u\in V_R$.
A combination with the above uniform 
estimates~\eqref{ineqvll0} and~\eqref{eq:thetaM.Linfty}
enable us to assert the existence of an element $\vartheta \in L^\infty({\mathbb T};H_R)\cap L^2({\mathbb T};V_R)$ with 
$\partial_t \vartheta \in 
L^{4/3}({\mathbb T};V_R')$,
and a sub-sequence $\{\vartheta^{(M')} \}$ of $\{ \vartheta^{(M)} \}_{M \in {\mathbb N}}$ such that
$$
\begin{aligned}
  \nabla \vartheta^{(M')} &\to \nabla \vartheta &&\textrm{ in } L^2({\mathbb T};L^2(\Omega_R)) \textrm{ weakly}, \\
  \vartheta^{(M')} &\to \vartheta &&\textrm{ in } L^\infty({\mathbb T};H_R) \textrm{ weakly-*}, \\
  \partial_t \vartheta^{(M')} &\to \partial_t \vartheta &&\textrm{ in }  L^{4/3}({\mathbb T};V_R') \textrm{ weakly}, \\
  \vartheta^{(M')} &\to \vartheta &&\textrm{ in } L^2({\mathbb T};L^{q_1}(\Omega)) \textrm{ strongly}, \quad 1 \leq q_1 < 6, \\
  \vartheta^{(M')}|_{\partial B_R} &\to \vartheta|_{\partial B_R}  &&\textrm{ in } L^2({\mathbb T};L^{q_2}(\partial B_R)) \textrm{ strongly}, \quad 1 \leq q_2 <  4,
\end{aligned}
$$
where the latter convergences follow from the Aubin--Lions Theorem and the embeddings \eqref{compac}. 
Passing to the limit $M' \to \infty$ in~\eqref{coefalfa}, with standard arguments, we find
that $v\coloneqq\vartheta+\tilde h$ satisfies
\begin{multline}
\int_{{\mathbb T}} (v(t),\psi'(t) \Psi  )_{\Omega_R}  \dt =  \int_{{\mathbb T}}  a(v(t),\psi(t) \Psi ) \dt  \label{weform}
 \\ 
+  \int_{{\mathbb T}}  c(v(t),v(t),\psi(t) \Psi ) \dt -  \int_{{\mathbb T}}  \left( f(t),\psi(t) \Psi  \right)_{\Omega_R} \dt, \, \forall \Psi  \in V_R, \, \forall \psi \in {\mathcal D}({\mathbb T}).
\end{multline}
Since the function spaces
 \[
\left\{\Phi :  \mathbb T \times  \Omega_R \to {\mathbb R}^3  :  \Phi (t,x) =  \psi (t)   \Psi (x), \,  \psi  \in {\mathcal D}(\mathbb T),  \,  \Psi   \in W_R \right\}, 
 \]
 \[
 \left\{ \eta:  \mathbb T \times  \Omega_R \to {\mathbb R} : \eta(t,x) = \psi (t)   \phi (x), \, \psi  \in {\mathcal D}(\mathbb T),  \,  \phi  \in L^2(\Omega_R) \right\}
 \]
are dense in $H^1(\mathbb T; W_R)$ and in $L^2(\mathbb T; L^2(\Omega_R))$, respectively, we obtain an equivalent definition of weak solution (in the velocity variable):
\begin{equation}
\label{eq:weak.td}
\begin{aligned}
&\int_{{\mathbb T}} (v(t),\Phi'(t) )_{\Omega_R}  \,\dt 
 =  \int_{{\mathbb T}}  a(v(t),\Phi(t)) \,\dt  
 \\
& \qquad +  \int_{{\mathbb T}}  c(v(t),v(t),\Phi(t) ) \,\dt -  \int_{{\mathbb T}}  \left( f(t),\Phi(t) \right)_{\Omega_R} \,\dt, \quad \forall \Phi  \in H^1({\mathbb T};V_R).
\end{aligned}
\end{equation}

The energy inequality~\eqref{enin} for $\vartheta$ is obtained from \eqref{ttoT} and integration over $\torus$, as
\[
\begin{aligned}
\int_{\torus} \| \vartheta^{(M')}\|^2_{(R,|\zeta|)}\, \dt = \, & \,  \int_{\torus} c(\vartheta^{(M')},\vartheta^{(M')},\tilde{h}) \,\dt +  \int_{\torus} \big( f  - \partial_t \tilde{h}, \vartheta^{(M')} \big)_{\Omega_R} \dt \\ 
& \, -  \, \int_{\torus} a(\tilde{h},\vartheta^{(M')})\, \dt -  \int_{\torus}  c(\tilde{h},\tilde{h},\vartheta^{(M')}) \,\dt,
\end{aligned}
\]
where we use the above convergence results for the subsequence $\{ \vartheta^{(M')} \}_{M' \in \mathbb N}$ to obtain, in particular,
\[
 \int_{\torus} c(\vartheta^{(M')},\vartheta^{(M')},\tilde{h}) \dt  \to  \int_{\torus} c(\vartheta,\vartheta,\tilde{h}) \dt, 
\]
and, by the lower semicontinuity of the norm,
\[
\begin{aligned}
\int_{\torus}\| \vartheta \|^2_{(R,|\zeta|)} \dt \leq  & \liminf_{M' \to \infty} \int_{\torus} \| \vartheta^{(M')}\|^2_{(R,|\zeta|)} \dt  \\ 
= & \int_{\torus} c(\vartheta,\vartheta,\tilde{h}) \dt +  \int_{\torus} \big( f  - \partial_t \tilde{h}, \vartheta \big)_{\Omega_R} \dt - \int_{\torus} a(\tilde{h},\vartheta) \dt -  \int_{\torus}  c(\tilde{h},\tilde{h},\vartheta) \dt \\
= & - \int_{\torus} c(\vartheta,\tilde{h},\vartheta) \dt +  \int_{\torus} \big( f  - \partial_t \tilde{h}, \vartheta \big)_{\Omega_R} \dt - \int_{\torus} a(\tilde{h},\vartheta) \dt -  \int_{\torus}  c(\tilde{h},\tilde{h},\vartheta) \dt \\
= & - \int_{\torus} c(v,\tilde{h},\vartheta) \dt +  \int_{\torus} \big( f  - \partial_t \tilde{h}, \vartheta \big)_{\Omega_R} \dt - \int_{\torus} a(\tilde{h},\vartheta) \dt,
\end{aligned}
\]
which is~\eqref{enin}.

We have solved problem \eqref{weform} in ${\mathcal D}'({\mathbb T};V'_R)$ for the velocity field. Our aim now  is to recover the pressure. For this purpose, we follow the ideas of \cite{JNeustupa}, and define ${\mathcal F} \in {\mathcal D}'({\mathbb T};W'_R)$ as follows:
\begin{multline*}
  \langle {\mathcal F}(\psi),\Psi  \rangle_{\Omega_R} := - \int_{{\mathbb T}} \left( v(t),\Psi \right)_{\Omega_R} \psi'(t) \, \dt  +  \int_{{\mathbb T}}  \langle {\mathcal A} v(t) + {\mathcal C}(v(t) ,v(t)), \Psi  \rangle_{\Omega_R} \psi(t) \, \dt    \\ 
  - \int_{{\mathbb T}}  \left( f(t),\Psi  \right)_{\Omega_R} \psi(t) \, \dt \quad ( \Psi  \in W_R,  \psi \in {\mathcal D}({\mathbb T})),
\end{multline*}
where now $\langle \cdot , \cdot  \rangle_{\Omega_R}$ represents the duality pairing between $W'_R$ and $W_R$. 

Let ${\mathbb P}_{V_R^\perp}$ be the projection operator from $W_R$ onto $V_R^\perp$, when considering the decomposition $W_R = V_R \oplus V_R^\perp$, orthogonal with respect to the inner product of $W_R$. Then ${\mathbb P}^*_{V_R^\perp}: (V_R^\perp)' \to W'_R$ and its range is given by 
$$
\text{Ran}({\mathbb P}^*_{V_R^\perp})=V_R^0 :=  \left\{ F \in W'_R: \, \left\langle F , u \right\rangle_{\Omega_R} = 0, \, \forall u \in V_R \right\} \cong (V_R^\perp)'.
$$
From the previous results for the velocity field, we have ${\mathcal F} \in {\mathcal D}'({\mathbb T};V_R^0)$. This means 
\[
{\mathcal F}  = {\mathbb P}^*_{V_R^\perp}{\mathcal F},
\]
and ${\mathbb P}^*_{V_R^\perp}{\mathcal F}$ is given by
\[
\begin{aligned}
  \langle {\mathbb P}^*_{V_R^\perp}{\mathcal F} (\psi),\Psi \rangle := &  - \int_{{\mathbb T}} \langle {\mathbb P}^*_{V_R^\perp} v(t),\Psi  \rangle_{\Omega_R} \psi'(t)  \,\dt  \\
&  +  \int_{{\mathbb T}}  \langle {\mathbb P}^*_{V_R^\perp} {\mathcal A} v(t) + {\mathbb P}^*_{V_R^\perp} {\mathcal C}(v(t) ,v(t)), \Psi  \rangle_{\Omega_R} \psi(t)  \,\dt    \\ 
&   - \int_{{\mathbb T}}  \langle {\mathbb P}^*_{V_R^\perp} f(t),\Psi  \rangle_{\Omega_R} \psi(t)  \,\dt \quad ( \Psi  \in W_R,  \psi \in {\mathcal D}({\mathbb T})).
\end{aligned}
\]
Consider  the operator $B: V_R^{\perp} \to L^2(\Omega_R)$ defined by
\[
\langle B v,p \rangle  = b(v,p) =   - \int_{\Omega_R} ( \nabla \cdot v ) p \, \dx \qquad (p\in L^2(\Omega)),
\]
which is an isomorphism. Then $B^*: L^2(\Omega_R)  \to V_R^0 = \text{Ran}({\mathbb P}^*_{V_R^\perp})$,
\[
\langle v, B^* p \rangle  = b(v,p) =   - \int_{\Omega_R} ( \nabla \cdot v ) p \, \dx \qquad (v\in V_R^\perp), 
\]
is also an isomorphism. Therefore, there exists $p_0 \in L^\infty({\mathbb T};L^2(\Omega))$, $p_1,p_3 \in L^2({\mathbb T};L^2(\Omega))$, $p_2 \in L^{4/3}({\mathbb T};L^2(\Omega))$ such that
\[
\begin{aligned}
\int_{{\mathbb T}} \langle {\mathbb P}^*_{V_R^\perp} v(t),\Psi  \rangle_{\Omega_R} \psi'(t)  \dt & = - \int_{{\mathbb T}} \int_{\Omega_R} p_0 (t,x) \psi'(t)  ( \nabla \cdot \Psi )(x)  \, \dx \dt , \\
\int_{{\mathbb T}}  \langle {\mathbb P}^*_{V_R^\perp} {\mathcal A} v(t) , \Psi  \rangle_{\Omega_R} \psi(t)  \dt & = - \int_{{\mathbb T}} \int_{\Omega_R}  p_1 (t,x) \psi(t)  ( \nabla \cdot \Psi )(x) \, \dx \dt , \\
\int_{{\mathbb T}}  \langle  {\mathbb P}^*_{V_R^\perp} {\mathcal C}(v(t) ,v(t)), \Psi  \rangle_{\Omega_R} \psi(t)  \dt  & = - \int_{{\mathbb T}} \int_{\Omega_R}  p_2 (t,x) \psi(t)  ( \nabla \cdot \Psi  )(x) \, \dx \dt  , \\
\int_{{\mathbb T}}  \langle {\mathbb P}^*_{V_R^\perp} f(t),\Psi  \rangle_{\Omega_R} \psi(t)  \dt & = - \int_{{\mathbb T}} \int_{\Omega_R} p_3 (t,x) \psi(t)  ( \nabla \cdot \Psi  )(x)  \, \dx \dt  .
\end{aligned}
\]
Hence
\[
\begin{aligned}
& - \int_{{\mathbb T}} \left( v(t),\Psi  \right)_{\Omega_R} \psi'(t)  \dt  -  \int_{{\mathbb T}} \left( p_0 (t) ,  \nabla \cdot \Psi \right)_{\Omega_R} \psi'(t)  \, \dx \dt   \\
& \quad +  \int_{{\mathbb T}} a(v(t),\Psi ) \psi(t)  \dt +   \int_{{\mathbb T}}  c(v(t),v(t),\Psi ) \psi(t)  \dt    \\ 
&\qquad = \int_{{\mathbb T}}  \left( f(t),\Psi  \right)_{\Omega_R} \psi(t)  \dt
- \int_{{\mathbb T}} \left(p_1 (t) + p_2(t) + p_3(t) , \nabla \cdot \Psi \right)_{\Omega_R} \psi(t)   \dx \dt ,
\end{aligned}
\]
for all $\Psi  \in W_R$ and $\psi \in {\mathcal D}(\mathbb T)$,
which shows~\eqref{eq:weak.form}.
This completes the existence proof.

Concerning uniqueness, let us suppose that, in addition to the weak solution $(v,p)$ already constructed, there exists a more regular solution $(\tilde v,\tilde p)$ as formulated in the theorem.
To derive an estimate of 
$\overline{v} := v-\tilde v$ 
in the norm $\| \cdot  \|_{(R,|\zeta|)}$ defined in~\eqref{eq:innerproduct.zeta},
we can argue as in the proof of Lemma~\ref{lem:relineq} below,
where we compare a weak solution with a strong solution in the exterior domain.
Instead of using the strong formulation and integrating by parts in space, 
we here employ the weak formulation for $(\tilde v,\tilde p)$.
Since~$\bop(\tilde v,\tilde p)=0$ on $\partial B_R$ in a weak sense,
several terms from the derivation of~\eqref{Is6} do not appear, 
and we arrive at
$$
\begin{aligned}
\int_\torus\| \overline v  \|_{(R,|\zeta|)}^2\,\dt
\leq &   \int_{{\mathbb T}\times \Omega_R} \overline{v} \cdot \nabla \overline{v} \cdot \tilde{v} \, \dx \dt  
 - 
\frac 12 \int_{{\mathbb T}\times \partial B_R} \left( \overline{v} \cdot \frac{x}{R}\right) (\tilde{v}  \cdot \overline{v}) \, \dS \dt  \\
= &  -  \int_{{\mathbb T}\times \Omega_R} \overline{v} \cdot \nabla \tilde{v} \cdot  \overline{v} \, \dx \dt  +
\frac{1}{2} \int_{{\mathbb T}\times \partial B_R} \left( \overline{v} \cdot \frac{x}{R}\right) (\tilde{v}  \cdot \overline{v}) \, \dS \dt .
\end{aligned}
$$
Then
\[
\int_\torus\| \overline v  \|_{(R,|\zeta|)}^2\,\dt
\leq \, C(\Omega_R)\big( \| \nabla \tilde{v} \|_{L^\infty({\mathbb T};L^2(\Omega_R))}  + \| \tilde{v} \|_{L^\infty({\mathbb T};L^2(\partial B_R))}  \big)\| \nabla \overline{v} \|^2_{L^2\left({\mathbb T} \times \Omega_R\right)} ,
\]
and $\| \nabla \overline{v} \|^2_{L^2\left({\mathbb T} \times \Omega_R\right)} =  \| \overline{v} \|_{L^2\left({\mathbb T}\times \partial B_R\right)} ^2= 0$ follows from the assumption~\eqref{eq:smallness.uniqueness} if $\delta>0$ is sufficiently small.
\end{proof}

\begin{rem} A similar uniqueness result can be established under the assumption 
 \[ \| \tilde{v} \|_{L^\infty({\mathbb T};L^3(\Omega_R))}  +  \| \tilde{v} \|_{L^\infty({\mathbb T};L^2(\partial B_R))}  \leq \delta,\] 
instead of~\eqref{eq:smallness.uniqueness},
in which case the last estimate is replaced with
$$
\begin{aligned}
\int_\torus\| \overline v  \|_{(R,|\zeta|)}^2\,\dt
\leq  C(\Omega_R) \left( \| \tilde{v} \|_{L^\infty({\mathbb T};L^3(\Omega_R))}  +  \| \tilde{v} \|_{L^\infty({\mathbb T};L^2(\partial B_R))}  \right) \| \nabla \overline{v} \|^2_{L^2\left({\mathbb T} \times \Omega_R\right)} .
\end{aligned}
$$
\end{rem}

\section{Estimates of the truncation error}
\label{error}

Consider the strong solution $(\uvel,\upres)$ 
to problem~\eqref{eq:nstp} in the exterior domain $\Omega$
and the weak solution $(\uvel_R,\upres_R)=(\vvel,\vpres)$
to problem~\eqref{eq:nstp.pert} in the truncated domain $\Omega_R$,
which were established in Theorem~\ref{thm:strongsol}
and Theorem~\ref{thm:weakR}, respectively.
In the following theorem we provide 
an estimate of the approximation error
under the assumption that the total flux $\Phi$ through $\partial\Omega$, 
defined in~\eqref{eq:flux},
is constant in time.

\begin{thm}
\label{thm:error}
Under the assumptions of Theorems  \ref{thm:strongsol} and \ref{thm:weakR},
and if $\ddt \Phi=0$,
there exist positive constants $C_i$, $i=0,1,2$, independent of $R$, such that if $\varepsilon \leq \sfrac{1}{C_0}$ then
\begin{equation}\label{eq:error.est}
\| \nabla u - \nabla u_R \|_{L^2({\mathbb T}\times \Omega_R)} 
 + \| u -  u_R \|_{L^2({\mathbb T}\times \partial B_R)} 
\leq    (C_1 \varepsilon + C_2 \varepsilon^2 ) \frac{1}{R^{\sfrac{1}{2}}}.
\end{equation}
\end{thm}

To prove Theorem~\ref{thm:error}, 
consider the error $(w,\wpres):= (\uvel,\upres) - (\uvel_R,\upres_R)$ associated with the approximation of $ (\uvel,\upres)$ by $(\uvel_R,\upres_R)$.
We measure this error in terms of the following inequality.
\begin{lem}
\label{lem:relineq}
The difference $w:= \uvel- \uvel_R$ satisfies
\begin{equation}
\begin{aligned}
&   \| \nabla w \|^2_{L^2\left({\mathbb T} \times \Omega_R\right)} + \frac{1}{R}  \| w \|^2_{L^2\left({\mathbb T}\times \partial B_R\right)} + \frac{|\zeta|}{2}  \| w \|^2_{L^2\left({\mathbb T}\times \partial B_R\right)}\\
&  \leq  \int_{{\mathbb T}\times \Omega_R} w \cdot \nabla w \cdot u \, \dx \, \dt  -  \int_{{\mathbb T}\times \partial B_R} \frac 12 \left(w \cdot \frac{x}{R}\right) (u  \cdot w) \,  \dS(x) \dt \\
& \quad  -  \int_{{\mathbb T}\times \partial B_R} \frac 12 \left(u \cdot \frac{x}{R}\right)  (u \cdot w)   \,  \dS(x)   \, \dt   
 + \int_{{\mathbb T}\times \partial B_R}   \frac 1R  \left( 1 + \wakefct(x)  \right) (u \cdot w) \,  \dS(x)   \, \dt \\
& \quad  +  \int_{{\mathbb T}\times \partial B_R}  \frac{x}{R} \cdot \nabla u \cdot w  \,  \dS(x)   \, \dt  - \int_{{\mathbb T}\times \partial B_R}     \upres 
 \left( \frac{x}{R} \cdot w \right)    \,  \dS(x)   \, \dt .
\end{aligned}
\label{Is6}
\end{equation}
\end{lem}

\begin{proof}
In what follows, we again consider the inner product $(\cdot , \cdot)_{(R;|\xi|)}$ in $H^1(\Omega_R)^3$ defined in~\eqref{eq:innerproduct.zeta},
and the multi-linear forms $a$ and $c$ defined in~\eqref{defab} and~\eqref{defc}, respectively.
Recall the notation $\vartheta := \uvel_R - \tilde{h}$,
and define $\mu := \uvel - \tilde{h} \in H^1(\mathbb T;V_R)$, so that $w = \uvel - \uvel_R = \mu - \vartheta$. Then, we have
\[
 \int_{\torus} \| w \|_{(R,|\zeta|)}^2 \dt  
=  \int_{\torus} ( \mu , w)_{(R,|\zeta|)} \dt -  \int_{\torus} (\vartheta, \mu)_{(R,|\zeta|)} \dt +  \int_{\torus} \| \vartheta \|_{(R,|\zeta|)}^2 \dt 
\]
Integration by parts in $\Omega_R$ and the fact that $(\uvel,\upres)$ 
is a strong solution to~\eqref{eq:nstp}
yield
\[
\begin{aligned}
 \int_{\torus} &(\mu,w)_{(R,|\zeta|)} \dt \\  
  = &  \int_{\torus \times \partial B_R} \left(  \frac 1R + \frac{|\zeta|}{2} \right)  (\mu \cdot w) \, \dS \dt 
  -  \int_{\torus} \left( \partial_t  u,  w \right)_{\Omega_R} \dt + \int_{\torus \times \Omega_R}  \zeta \cdot \nabla u  \cdot w \, \dx \dt \\
  & -  \int_{\torus \times \Omega_R} u \cdot \nabla u  \cdot w  \, \dx \dt  +  \int_{\torus} ( f, w)_{\Omega_R} \dt  - \int_{\torus \times \Omega_R} \nabla \tilde{h} : \nabla w\, \dx \dt  \\
  & +  \int_{\torus \times \partial B_R} \frac xR \cdot \nabla u \cdot w \, \dS(x) \dt  - \int_{\torus \times \partial B_R} \frac xR \cdot  w \upres \, \dS(x) \dt .
\end{aligned}
\]
We next take the test function $\Phi=\mu=u-\tilde h$ in the weak formulation~\eqref{eq:weak.td},
which is admissible 
since $\mu\in L^2(\torus; V_R)$,
$\tilde h\in H^1(\torus;H^1(\Omega)^3)$
and $u\in H^1(\torus;L^p(\Omega_R)^3)$ for any $p\in(1,\infty)$.
Decomposing $\vartheta=u_R-\tilde h$, we get
\[
\begin{aligned}
- \int_{\torus} (\vartheta,\mu)_{(R,|\zeta|)} \dt  
= &- \int_{\torus}(u_R,\partial_t\mu)\,\dt
+\int_{\torus} c(u_R,u_R,\mu)\,\dt
-\int_{\torus}(f,\mu)_{\Omega_R}\,\dt 
+\int_{\torus} a(\tilde h,\mu)\,\dt
\\
&- \int_{\torus}  \int_{\Omega_R}  \zeta \cdot \nabla \vartheta  \cdot \mu \, \dx \dt + \int_{\torus}  \int_{\partial B_R}  \frac 1R  \frac{(\zeta \cdot x)}{2} (\vartheta \cdot \mu) \, \dS(x) \dt .
\end{aligned}
\]
Since
$\vartheta = \uvel_R - \tilde{h}$ satisfies the energy inequality \eqref{enin}, we further have 
\[
\int_{\torus} \| \vartheta \|_{(R,|\zeta|)}^2 \dt
\leq  \int_{\torus} \big( f  - \partial_t \tilde{h}, \vartheta \big)_{\Omega_R}  \dt
- \int_{\torus} c(u_R,\tilde h,\vartheta)\,\dt - \int_{\torus}a(\tilde h,\vartheta) \,\dt.
\]
To combine the terms in the above expressions,
we take into account the identities
\[
\begin{aligned}
(f,w)_{\Omega_R} - (f,\mu)_{\Omega_R}+(f,\vartheta)_{\Omega_R}&=0,
\end{aligned}
\]
as well as
\[
\begin{aligned}
\int_\torus \big[-\left( \partial_t  u,  w \right)_{\Omega_R}-  (u_R,\partial_t\mu)_{\Omega_R} 
&- (\partial_t \tilde h,\vartheta)_{\Omega_R}\big]\,\dt
=\int_\torus \big[ -\left( \partial_t  u,  u \right)_{\Omega_R}+(\tilde h,\partial_t \tilde h)_{\Omega_R}\big]\,\dt
\\
&=  \int_\torus \frac{\mathrm d}{\dt}\int_{\Omega_R}\big[-\frac{1}{2}|u|^2+\frac{1}{2}|\tilde h|^2\big]\,\dx\dt = 0.
 \end{aligned}
\]
Due to the identity 
\[
\begin{aligned}
&\int_{\Omega_R}  \zeta \cdot \nabla u  \cdot w \, \dx
- \int_{\Omega_R} \nabla \tilde{h} : \nabla w\, \dx 
+a(\tilde h,\mu)-a(\tilde h,\vartheta)
\\
&\qquad
= \int_{\Omega_R}  \zeta \cdot \nabla \mu  \cdot w \, \dx
+ \int_{\partial B_R}\frac{1}{R}(1+\wakefct(x))\tilde h\cdot w\,\dS(x),
\end{aligned}
\]
we can further collect the terms related with $\zeta$ as
\[
\begin{aligned}
&\int_{\partial B_R} \left(  \frac 1R + \frac{|\zeta|}{2} \right)  (\mu \cdot w) \, \dS
+\int_{\Omega_R}  \zeta \cdot \nabla u  \cdot w \, \dx
- \int_{\Omega_R} \nabla \tilde{h} : \nabla w\, \dx 
+a(\tilde h,\mu) 
\\
&\qquad\qquad
- \int_{\Omega_R}  \zeta \cdot \nabla \vartheta  \cdot \mu \, \dx + \int_{\partial B_R}  \frac 1R  \frac{(\zeta \cdot x)}{2} (\vartheta \cdot \mu) \, \dS(x)
-a(\tilde h,\vartheta)
\\
&
= \int_{\partial B_R}\frac{1}{R}(1+\wakefct(x))(u\cdot w)\,\dS(x)
- \int_{\partial B_R}  \frac 1R  \frac{(\zeta \cdot x)}{2} (w \cdot \mu) \, \dS(x)
+\int_{\Omega_R}  \zeta \cdot \nabla \mu  \cdot w \, \dx
\\
&\qquad\qquad
+ \int_{\Omega_R}  \zeta \cdot \nabla \mu  \cdot \vartheta \, \dx 
- \int_{\partial B_R}  \frac 1R  \frac{(\zeta \cdot x)}{2} (\vartheta \cdot \mu) \, \dS(x)
\\
&= \int_{\partial B_R}\frac{1}{R}(1+\wakefct(x))(u\cdot w)\,\dS(x),
\end{aligned}
\]
where we used integration by parts and that $\mu=w+\vartheta$.
Recalling the property~\eqref{cigual0} of $c$,
we further have
\[
\begin{aligned}
&c(u_R,u_R,\mu)-c(u_R,\tilde h,\vartheta)
\\ 
&\quad
=c(u_R,u_R,\mu)-c(u_R,u_R,\vartheta)+c(u_R,\vartheta,\vartheta)
=c(u-w,u-w,w)
\\
&\quad
= c(u,u,w)-c(w,u,w)-c(u,w,w)+c(w,w,w)
=c(u,u,w)+c(w,w,u).
\end{aligned}
\]
In this way, we arrive at
\[
\begin{aligned}
&\int_{\torus} \| w \|_{(R,|\zeta|)}^2 \dt
\\
&\leq 
-  \int_{\torus \times \Omega_R} u \cdot \nabla u  \cdot w  \, \dx \dt   
+  \int_{\torus \times \partial B_R} \frac xR \cdot \nabla u \cdot w \, \dS(x) \dt  - \int_{\torus \times \partial B_R} \frac xR \cdot  w \upres \, \dS(x) \dt 
  \\
&\quad
+\int_{\torus} c(u,u,w)\,\dt
-\int_\torus c(w,w,u)\,\dt
+\int_{\torus\times\partial B_R}\frac{1}{R}(1+\wakefct(x))(u\cdot w)\,\dS(x)\dt.
\end{aligned}
\]
Invoking the definition of $c$, see~\eqref{defc}, 
we conclude~\eqref{Is6}.
\end{proof}

With inequality~\eqref{Is6} at hand,
we now show that the velocity error $\wvel$ tends to zero in appropriate norms. It is useful to recall the properties of strong solutions in the exterior domain, as outlined in Remark \ref{fluxcte}. 

\begin{proof}[Proof of Theorem~\ref{thm:error}]
In order to estimate $\wvel$, we use Lemma~\ref{lem:relineq}
and write~\eqref{Is6} as
\[
\begin{aligned}
 & \int_\torus\| w \|^2_{(R,|\zeta|)}\,\dt
 \leq  \int_{{\mathbb T}\times \Omega_R} w \cdot \nabla w \cdot u \, \dx \, \dt  +  \int_{{\mathbb T}\times \partial B_R}\left[ - \frac 12 \left(w \cdot \frac{x}{R}\right) \right] (u  \cdot w) \,  \dS(x)\, \dt  \\
& \quad  +  \int_{{\mathbb T}\times \partial B_R} \left[ - \frac 12 \left(u \cdot \frac{x}{R}\right) \right] (u \cdot w)   \,  \dS(x)  \, \dt 
 + \int_{{\mathbb T}\times \partial B_R}   \frac 1R  \left( 1 + \wakefct(x)  \right) (u \cdot w) \,  \dS(x)   \, \dt \\
& \quad  +  \int_{{\mathbb T}\times \partial B_R}  \frac{x}{R} \cdot \nabla u \cdot w  \,  \dS(x)   \, \dt  + \int_{{\mathbb T}\times \partial B_R}   \left[  -  \upres 
 \left( \frac{x}{R} \cdot w \right)   \right] \,  \dS(x)   \, \dt =: \sum_{i =1}^6 I_j,
\end{aligned}
\]
and we estimate $I_1$, \dots, $I_6$ separately.

Take a fixed $S \in (0,\infty)$ with $\partial B_S\subset\Omega$. Let $R > S$. 
From Lemma~\ref{estdk} and estimate~\eqref{estssol} in Theorem~\ref{thm:strongsol}, we obtain
$$
\begin{aligned}
\left| \int_{{\mathbb T} \times B_R \setminus \overline{B_S}} w \cdot \nabla w \cdot u \, \dx \dt  \right| \, \leq &  \int_{{\mathbb T} \times B_R \setminus \overline{B_S}} \frac{|w(t,x)|}{|x|}|\nabla w(t,x)| |x| |u(t,x)| \, \dx \dt   \\
 \leq & \, C(S) \| w \|_{({\mathbb T},R)}^2 \| u \|_{\infty,\nu^1;\torus\times B^S} 
\end{aligned}
$$
and therefore, by Poincar\'e's and H\"older's inequalities,
$$
\begin{aligned}
  I_1 
 \leq &  \left|  \int_{{\mathbb T}\times \Omega_S} w \cdot \nabla w \cdot u \, \dx \dt \right| + \left| \int_{{\mathbb T} \times B_R \setminus \overline{B_S}} w \cdot \nabla w \cdot u \, \dx \dt \right|
  \\
\leq & \, C(S,\partial \Omega) \|\nabla w \|^2_{L^2({\mathbb T} \times \Omega_S)} \| u \|_{L^\infty(\mathbb T;L^3(\Omega_S))} + 
C(S)\| u \|_{\infty,\nu^1;\torus\times B^S}  \| w \|_{({\mathbb T},R)}^2   \\
\leq &  \, C(S,\partial \Omega)  \left[  \| u \|_{L^\infty(\mathbb T;L^3(\Omega_S))} + \| u \|_{\infty,\nu^1;\torus\times B^S} \right] \| w \|_{({\mathbb T},R)}^2
  \\
\leq &  \, C(S,\partial \Omega) \, \varepsilon \| w \|_{({\mathbb T},R)}^2 .
\end{aligned}
$$
From \eqref{estssol}, we also get the following estimates for the integrals over ${\mathbb T}\times \partial B_R$ involving the velocity $u$: 
\[
\begin{aligned}
I_2 \,
 \leq & \frac 12 \int_{{\mathbb T}\times \partial B_R} \frac{1}{R}\left| w(t,x) \right|^2  |x| |u(t,x)| \, \dS(x) \dt \\
  \leq & \,  \frac 12 \| u \|_{\infty,\nu^1;\torus\times B^S} \frac{1}{R}  \| w \|^2_{L^2({\mathbb T}\times \partial B_R)} 
  \leq  \varepsilon \| w \|_{({\mathbb T},R)}^2, 
\end{aligned}
\]
and analogously,
\[
\begin{aligned}
 I_3 
& \leq  \frac{1}{2 R^{3/2}}\int_{{\mathbb T}\times \partial B_R} |x|^2 |u(t,x)|^2 \frac{|w(t,x)|}{R^{1/2}} \, \dS(x) \dt    \\
& \leq  \frac{C}{R^{1/2}}  \| u \|^2_{\infty,\nu^1;\torus\times B^S}  \| w \|_{({\mathbb T},R)} 
 \leq  \frac{C}{R^{1/2}} \varepsilon^2 \| w \|_{({\mathbb T},R)}
\end{aligned}
\]
and
\[
\begin{aligned}
I_4 
 & \leq \frac{1}{R^{3/2}} \int_{{\mathbb T}\times \partial B_R} |x| \left(1 + \wakefct(x) \right) |u(t,x)| \frac{ |w(t,x)|}{R^{1/2}} \, \dS(x) \dt   \\
& \leq \frac{C}{R^{1/2}}  \| u \|_{\infty,\nu^1_1(\cdot;\zeta);\torus\times B^S
}  \| w \|_{({\mathbb T},R)} \leq \frac{C}{R^{1/2}} \varepsilon \| w \|_{({\mathbb T},R)}.
\end{aligned}
\]
From \eqref{far}, we obtain
$
{\mathcal J}_R(3,3) 
\leq C R^{-2}.
$
Combined with estimates \eqref{estssol}, this yields
\[
\begin{aligned}
I_5
& \leq  R^{1/2} \| \nabla u \|_{\infty,\nu^{3/2}_{3/2}(\cdot;\zeta);\torus\times B^S}  {\mathcal J}_R(3,3)^{1/2} \frac{\| w \|_{L^2\left({\mathbb T}\times \partial B_R\right)}}{R^{1/2}}  \\
& \leq   C \frac{1}{R^{1/2}} \| \nabla u \|_{\infty,\nu^{3/2}_{3/2}(\cdot;\zeta);\torus\times B^S} \| w \|_{({\mathbb T},R)} \leq   C \frac{1}{R^{1/2}}\varepsilon \| w \|_{({\mathbb T},R)}.
\end{aligned}
\]
Finally, the term with the pressure $\upres$ is estimated as
\[
\begin{aligned}
I_6
& \leq   \frac{1}{R^{3/2}} \int_{{\mathbb T}\times \partial B_R} |x|^2  |\upres(t,x)| 
\frac{|w(t,x)|}{R^{1/2}} \, \dS(x) \dt 
 \\
& \leq   \frac{C}{R^{3/2}} \| \upres  \|_{\infty,\nu^2;\torus\times B^S}  R  \frac{\| w \|_{L^2({\mathbb T}\times \partial B_R)}}{R^{1/2}}
\leq   \frac{C}{R^{1/2}} \varepsilon \| w \|_{({\mathbb T},R)}
\end{aligned}
\]
by~\eqref{eq:pres.constantflux}, which holds due to $\ddt \Phi=0$.

In summary, we find
\[
(1 - C_0 \varepsilon) \| w \|_{({\mathbb T},R)}^2 +  \frac{|\zeta|}{2} \| w \|^2_{L^2\left({\mathbb T}\times \partial B_R\right)} 
 \leq     \frac{1}{R^{1/2}}  (C_1 \varepsilon + C_2 \varepsilon^2 )  \| w \|_{({\mathbb T},R)}.
\]
If $\left( 1 - C_0 \varepsilon \right) > 0$,
then (redefining the constants)
\[
\| w \|_{({\mathbb T},R)} +  \| w \|_{L^2\left({\mathbb T}\times \partial B_R\right)} \leq C \left( \| w \|_{({\mathbb T},R)} + \sqrt{\frac{|\zeta|}{2}}  \| w \|_{L^2\left({\mathbb T}\times \partial B_R\right)} \right)
 \leq     \frac{1}{R^{1/2}}  (C_1 \varepsilon + C_2 \varepsilon^2 ),
\]
which gives~\eqref{eq:error.est} and concludes the proof.
\end{proof}

\begin{rem}
For the convergence statement of Theorem~\ref{thm:error},
we had to assume $\ddt \Phi=0$,
that is, that the total flux through the boundary is constant in time.
As shown in Theorem~\ref{thm:strongsol}
this condition ensures that the decay rate of the pressure is $|x|^{-2}$, 
compare Remark~\ref{fluxcte}.
In the previous proof, this lead to a suitable estimate of the term $I_6$,
which cannot be obtained form the weaker rate $|x|^{-1}$
that holds in the general case.
\end{rem}

\subsection*{Acknowledgments}
The research of Thomas Eiter has been funded by Deutsche Forschungsgemeinschaft (DFG) through grant CRC 1114 ``Scaling Cascades in Complex Systems'', Project Number 235221301, Project YIP. Ana L. Silvestre acknowledges the financial support of Funda\c{c}\~ao para a Ci\^encia e a Tecnologia  (FCT), Portuguese Agency for Scientific Research, through the project UIDB/04621/2025 of CEMAT/IST-ID.


\end{document}